\documentclass[final,onefignum,onetabnum]{siamart190516}

\usepackage{braket,amsfonts}
\usepackage{multirow}
\usepackage{array}
\usepackage{bm}
\usepackage{amssymb,amsmath}
\usepackage{float}
\usepackage{algorithm}
\usepackage{algorithmic} 
\usepackage{graphics} 
\usepackage{subcaption}
\usepackage{multirow}
\usepackage{url}
\usepackage{mathdots}
\usepackage{arydshln}
\usepackage{bm}
\usepackage{ntheorem}

\usepackage{pgfplots}

\newsiamthm{claim}{Claim}
\newsiamremark{remark}{Remark}
\newsiamremark{hypothesis}{Hypothesis}
\crefname{hypothesis}{Hypothesis}{Hypotheses}

\usepackage{algorithmic}

\usepackage{graphicx,epstopdf}
\usepackage{amsmath}    % AMSLaTeX宏包 用来排出更加漂亮的公式
\usepackage{amssymb}  
\usepackage{algorithm}
\usepackage{algorithmic} 
\Crefname{ALC@unique}{Line}{Lines}

\usepackage{amsopn}

\usepackage{xspace}
\usepackage{bold-extra}
\usepackage[most]{tcolorbox}

\colorlet{texcscolor}{blue!50!black}
\colorlet{texemcolor}{red!70!black}
\colorlet{texpreamble}{red!70!black}
\colorlet{codebackground}{black!25!white!25}

\lstdefinestyle{siamlatex}{%
  style=tcblatex,
  texcsstyle=*\color{texcscolor},
  texcsstyle=[2]\color{texemcolor},
  keywordstyle=[2]\color{texemcolor},
  moretexcs={cref,Cref,maketitle,mathcal,text,headers,email,url},
}

\tcbset{%
  colframe=black!75!white!75,
  coltitle=white,
  colback=codebackground, % bottom/left side
  colbacklower=white, % top/right side
  fonttitle=\bfseries,
  arc=0pt,outer arc=0pt,
  top=1pt,bottom=1pt,left=1mm,right=1mm,middle=1mm,boxsep=1mm,
  leftrule=0.3mm,rightrule=0.3mm,toprule=0.3mm,bottomrule=0.3mm,
  listing options={style=siamlatex}
}

\newtcblisting[use counter=example]{example}[2][]{%
  title={Example~\thetcbcounter: #2},#1}

\newtcbinputlisting[use counter=example]{\examplefile}[3][]{%
  title={Example~\thetcbcounter: #2},listing file={#3},#1}

\DeclareTotalTCBox{\code}{ v O{} }
{ %fontupper=\ttfamily\color{texemcolor},
  fontupper=\ttfamily\color{black},
  nobeforeafter,
  tcbox raise base,
  colback=codebackground,colframe=white,
  top=0pt,bottom=0pt,left=0mm,right=0mm,
  leftrule=0pt,rightrule=0pt,toprule=0mm,bottomrule=0mm,
  boxsep=0.5mm,
  #2}{#1}

\patchcmd\newpage{\vfil}{}{}{}
\begin{tcbverbatimwrite}{tmp_\jobname_header.tex}
\title{Sparse Sampling Kaczmarz-Motzkin Method with Linear Convergence}

\author{Ziyang Yuan\thanks{Mathematical Department, National University of Defense Technology (\email{yuanziyang11@nudt.edu.cn}).}
\and Hui Zhang\thanks{Corresponding Author, Mathematical Department, National University of Defense Technology ( \email{h.zhang1984@163.com})}\and Hongxia Wang\thanks{Mathematical Department, National University of Defense Technology ( \email{wanghongxia@nudt.edu.cn})}.}

\headers{Greedy randomized sparse Kaczmarz method with linear convergence}{Ziyang Yuan, Hui Zhang, Hongxia Wang}
\end{tcbverbatimwrite}
\input{tmp_\jobname_header.tex}

\ifpdf
\hypersetup{ pdftitle={Guide to Using  SIAM'S \LaTeX\ Style} }
\fi

\begin{document}
\maketitle

%% ------------------------------------------------------------------
%% ABSTRACT
%% ------------------------------------------------------------------
\begin{tcbverbatimwrite}{tmp_\jobname_abstract.tex}
\begin{abstract}
The randomized sparse Kaczmarz method was recently proposed to recover sparse solutions of linear systems. In this work, we introduce a greedy variant of the randomized sparse Kaczmarz method by employing the sampling Kaczmarz-Motzkin method, and prove its linear convergence in expectation with respect to the Bregman distance in the noiseless and noisy cases. This greedy variant can be viewed as a unification of the sampling Kaczmarz-Motzkin method and the randomized sparse Kaczmarz method, and hence inherits the merits of these two methods. Numerically, we report a couple of experimental results to demonstrate its superiority.
\end{abstract}


\end{tcbverbatimwrite}
\input{tmp_\jobname_abstract.tex}
%% ------------------------------------------------------------------
%% END HEADER
%% ------------------------------------------------------------------

\section{Introduction}
\label{sec:intro}
The Kaczmarz method, originally appeared in \cite{Kaczmarz1937Angenaherte},  might be the most well-known method for finding an approximation solution to large-scale linear systems of the form
\begin{eqnarray}
\mathbf{A}\mathbf{x} = \mathbf{b},\label{problem}  
\end{eqnarray}
where $\mathbf{A}\in\mathbb{R}^{m\times n}$
and $\mathbf{b}\in\mathbb{R}^m$.  It has been utilized in a large range of applications such as computer tomography \cite{1970Algebraic}, digital signal processing \cite{1987Sezan}, distributed computing \cite{censor2002the} and many other engineering and physics problems. 
The standard Kaczmarz iterative scheme reads as  
\begin{eqnarray}\label{framework}
\mathbf{x}_{k+1}=\mathbf{x}_{k}- \frac{\langle\mathbf{a}_{i}, \mathbf{x}_{k}\rangle-b_{i}}{\left\|\mathbf{a}_{i}\right\|_{2}^{2}} \mathbf{a}_{i},
\end{eqnarray}
where $\mathbf{a}_i$ is the vector transposed by the $i$-th row of $\mathbf{A}$, $b_i$ is the $i$-th entry of $\mathbf{b}$, and the index $i$ is chosen cyclically. Geometrically, (\ref{framework}) says that $\mathbf{x}_{k+1}$ is obtained by projecting $\mathbf{x}_k$ onto the hyperplane $\{\mathbf{x}\in\mathbb{R}^n:\langle\mathbf{a}_i,\mathbf{x}\rangle=b_i\}$. When it is initialized with $\mathbf{x}_0 = \bm{0}$, the iteration sequence $\{\mathbf{x}_k\}$ generated by the Kaczmarz method converges linearly to the minimum 2-norm solution $\hat{\mathbf{x}}$ of $\mathbf{A}\mathbf{x} = \mathbf{b}$. However, the rate of convergence is hard to estimate. In 2009, the authors of
\cite{2009AVershyin} first analyzed a randomized variant of the Kaczmarz method. Instead of choosing $i$ cyclically, it updates $\mathbf{x}_k$ via (\ref{framework}) at random by choosing the $i$-th row with probability $\frac{\|\mathbf{a}_i\|^2_2}{\|\mathbf{A}\|^2_{\mathrm{F}}}$. Theoretically, they showed that the randomized Kaczmarz method converges linearly in the sense that 
\begin{equation}\label{lin}
\mathbb{E}\left[\left\|\mathbf{x}_{k+1}-\hat{\mathbf{x}}\right\|_{2}^{2}\right] \leq\left(1-\frac{\sigma^2_{\min}(\mathbf{A})}{\|\mathbf{A}\|_{\text{F}}^2}\right) \cdot \mathbb{E}\left[\left\|\mathbf{x}_{k}-\hat{\mathbf{x}}\right\|_{2}^{2}\right],
\end{equation}
where 
$\|\mathbf{A}\|_{\text{F}}$ is the Frobenius norm and $\sigma_{\min}(\mathbf{A})$ denotes the smallest singular value of $\mathbf{A}$. This elegant result triggers a great of researches into developing new variants and corresponding convergence analysis; seeing \cite{briskman2015block,leventhal2010randomized,needell2014paved,7032145,eldar2011acceleration,needell2014stochastic}. Recently, finding sparse solutions of linear systems becomes a popular topic in data science and machine learning. In this paper, we focus on another variant of the Kaczmarz method namely the randomized sparse Kaczmarz method, which was developed in
\cite{lorenz2014the,7025269,2016Linear}. Specifically, the iterative procedure of randomized sparse Kaczmarz method can be formulated as
\begin{eqnarray}\label{SRK}
\begin{aligned}
\mathbf{x}_{k+1}^{*} &=\mathbf{x}_{k}^{*}-\frac{\left\langle \mathbf{a}_{i}, \mathbf{x}_{k}\right\rangle-\mathbf{b}_{i}}{\left\|\mathbf{a}_{i}\right\|_{2}^{2}} \cdot \mathbf{a}_{i}, \\
\mathbf{x}_{k+1} &=S_{\lambda}\left(\mathbf{x}
_{k+1}^{*}\right),
\end{aligned}
\label{sparse}
\end{eqnarray}
where the parameter $\lambda$ is a nonnegative parameter, and the index $i$ is chosen by the same probability of the randomized Kaczmarz method. The soft thresholding operator $S_{\lambda}(x)=\textrm{max}\{|x|-\lambda,0\}\cdot \textrm{sign}(x)$ is introduced here to generate sparse solutions. In the case of $\lambda=0$, $S_\lambda(\cdot)$ reduces to the identity operator so that $\mathbf{x}^*_{k+1}=\mathbf{x}_{k+1}$ and hence the randomized sparse Kaczmarz method generalizes randomized Kaczmarz method in \cite{2009AVershyin}. In the case of $\lambda\neq0$, a larger $\lambda$ is likely to generate a sparser solution. Interestingly, it has also been shown that randomized sparse Kaczmarz method still converges linearly in a similar way to \eqref{lin}\cite{2016Linear}. 

As a natural development, we wonder whether randomized sparse Kaczmarz method still works well, or even better, if it is equipped with more advanced sampling strategies that may accelerate the convergence rate. To this end, by employing the sampling Kaczmarz-Motzkin method, which essentially combines the random and greedy ideas together, we propose a new variant, called sparse sampling Kaczmarz-Motzkin method (SSKM). The proposed variant can be viewed as a blender of randomized sparse Kaczmarz method and the sampling Kaczmarz-Motzkin method. SSKM method randomly samples $\beta$ rows from $\mathbf{A}$, and then
greedily picks out the most-violated one among the $\beta$ rows(refer to the second column of Table \ref{convergence rate}). Hence, the SSKM method inherits the merits of these two methods. Theoretically, by introducing the concept of Bregman projection, we prove that the SSKM method converges linearly in expectation in the noiseless case. Especially, as listed in the third column of Table \ref{convergence rate}, it can have a faster convergence rate comparing to the previous results. Furthermore, we also show that the same linear convergence rate can be held even with noisy observed data. Finally, we demonstrate the superiority of SSKM method by groups of numerical experiments. 

The paper is organized as follows. Section 2 shows some theories about Bregman projection. In section 3, we introduce SSKM method and prove its linearly convergence. Section 4 reports the numerical tests. Section 5 is the conclusion.

%\subsection{sampling}
%chooses the most violated term from a subset of rows in $\mathbf{A}$, namely choosing $\mathbf{a}_{t}$ so that $t\in\arg max_{i\in{\tau}}\|\mathbf{a}_i^{T}\mathbf{x}-\mathbf{b}_i\|$, $\tau\subset [m]$, $|\tau|=\beta$.
%
%
%Under the sampling rule below which is given in \cite{haddock2019greed},
%\begin{eqnarray}
%p_{\mathbf{x}}(\tau)=\frac{\|\mathbf{a}_{t(\tau_k,\mathbf{ x})}\|^2}{\sum_{\tau\in\binom{[m]}{\beta}}\|\mathbf{a}_{t(\tau,\mathbf{ x})}}\|^2, \label{sampling rule}
%\end{eqnarray}
%we prove SSKM method can have a linear convergence rate. Moreover, from Table \ref{convergence rate}, we can also find that SSKM method not only accelerate the convergence rate of SRK but also recover the rate of RK method, SKM method and Motzkin method(MM), which is a special case of SKM when $\beta=m$.  The is controlled by the dynamics parameter $\gamma_k$, which is defined as below,
%\begin{eqnarray*}
%	\gamma_k=\frac{\sum_{\tau_k\in\binom{[m]}{\gamma_k}}^{}\|\mathbf{A}\mathbf{x}_k-\mathbf{ b}\|^2}{\sum_{\tau_k\in\binom{[m]}{\gamma_k}}^{}\|\mathbf{A}\mathbf{x}_k-\mathbf{ b}\|^2_{\infty}},
%\end{eqnarray*}
%Note that $1\leq\gamma_k\leq\beta_k$. So The SSKM method is actually a more generalization and flexible framework. At the same time, numerical tests also demonstrate the superiority of SSKM method and can have a faster convergence rate. 
\begin{table}
	\caption{Convergence rate comparison between different methods. The rows of $\mathbf{ A}$ are normalized, $\|\mathbf{ e}_k\|^2_2=\|\mathbf{A}\mathbf{x}^k-\mathbf{b}\|^2_2$,  $|\hat{\mathbf{x}}|_{\min}$ means the smallest nonzero absolute element, $\tilde{\sigma}_{\min}(\mathbf{A})$ is the non-zero smallest singular value, and $\beta_k/\gamma_k\geq 1$.}
	\centering
	\label{convergence rate}
	\begin{tabular}{|c|c|c|}
		\hline & \text { Selection Rule } & \text { Convergence Rate } \\
		\hline $\mathrm{RK}\cite{2009AVershyin}$ &$ \mathbb{P}\left(t_{j}=i\right)=\frac{\left\|\mathbf{a}_{i}\right\|^{2}_2}{\|\mathbf{A}\|_{\text{F}}^{2}}$ &$ \mathbb{E}\left\|\mathbf{e}_{k}\right\|^{2} \leq\left(1-\frac{\sigma_{\min }^{2}(\mathbf{A})}{\|\mathbf{A}\|_{\text{F}}^{2}}\right)^{k}\left\|\mathbf{e}_{0}\right\|^{2}_2$ \\
		\hline 
		$\mathrm{SRK}$\cite{2016Linear}
		&$\mathbb{P}\left(t_{j}=i\right)=\frac{\left\|\mathbf{a}_{i}\right\|^{2}_2}{\|\mathbf{A}\|_{\text{F}}^{2}}$& 
		$\mathbb{E}\left\|\mathbf{e}_{k}\right\|^{2} \leq\left(1-\frac{\tilde{\sigma}_{\min }^{2}(A)}{2m}\cdot\frac{|\hat{\mathbf{ x}}|_{\min}}{|\hat{\mathbf{ x}}|_{\min}+2\lambda}\right)^{k}\left\|\mathbf{e}_{0}\right\|^{2}_2$ 
		\\
		\hline 
			\multirow{2}{*}{$\mathrm{SSKM}$}
	 & $\tau_j\sim\binom{[m]}{\beta}$&	 \multirow{2}{*}{$\mathbb{E}\left\|\mathbf{e}_{k}\right\|^{2} \leq\prod_{i=0}^{k}\left(1-\frac{\beta_k\tilde{\sigma}_{\min }^{2}(A)}{2\gamma_km}\cdot\frac{|\hat{\mathbf{ x}}|_{\min}}{|\hat{\mathbf{ x}}|_{\min}+2\lambda}\right)\left\|\mathbf{e}_{0}\right\|^{2}_2$ }\\
	 & $t_{j}=\arg \max _{i\in\tau_j}(\mathbf{a}_{i}^{\top} \mathbf{x}_{j-1}-b_{i})^2$ &\\
		\hline
	\end{tabular}
\end{table}

\section{Preliminaries}
First, we recall some concepts and properties of convex functions.
\subsection{Basic notions}
Let $f:\mathbb{R}^n\rightarrow\mathbb{R}$ be convex. Define the subdifferential of $f$ at $\mathbf{x}\in\mathbb{R}^n$ by
$$\partial f(\mathbf{x}):=\{\mathbf{x}^*\in\mathbb{R}^n:f(\mathbf{y})\geq f(\mathbf{x})+\langle\mathbf{x}^*,\mathbf{y}-\mathbf{x}\rangle,\forall\mathbf{y}\in\mathbb{R}^n\}.$$
Each $\mathbf{x}^*$ is called a subgradient of $f$ at $\mathbf{x}$. Next, let us define the strong convexity.
\begin{definition}
	The function $f:\mathbb{R}^n\rightarrow\mathbb{R}$ is said to be strongly convex, if there exists $\alpha>0$, so that for any $\mathbf{x},\mathbf{y}\in\mathbb{R}^n$ and $\mathbf{x}^*\in\partial f(\mathbf{x})$, we have
	$$f(\mathbf{y})\geq f(\mathbf{x})+\langle\mathbf{x}^*,\mathbf{y}-\mathbf{x}\rangle+\frac{\alpha}{2}\|\mathbf{x}-\mathbf{y}\|_2^2.$$
\end{definition}

\noindent If the concrete value of $\alpha$ is involved, then $f$ is said to be $\alpha$-strongly convex. The Fenchel conjugate $f^*$ of  $f$ is given by,
$$f^*(\mathbf{x}):=\sup _{\mathbf{z} \in\mathbb{R}^n}\{\langle\mathbf{x},\mathbf{z}\rangle-f(\mathbf{z})\big\}.$$
There are many interesting connections between $f$ and $f^*$. Especially, as illustrated by the following fact, the strong convexity of $f$ can imply the smoothness of $f^*$, 
\begin{theorem}[\cite{rockafellar2009variational}]
If $f:\mathbb{R}^n\rightarrow\mathbb{R}$ is $\alpha$-strongly convex, then the Fenchel conjugate function $f^*$ is differentiable with $1/\alpha$-Lipschitz continuous gradient, that is
$$\|\nabla f^*(\mathbf{x})-\nabla f^*(\mathbf{y})\|_2\leq\frac{1}{\alpha}\cdot\|\mathbf{x}-\mathbf{y}\|_2,\forall\mathbf{x},\mathbf{y}\in\mathbb{R}^n.$$
\end{theorem}

%Next, the concept of Bregman distance will be given, which has been successfully applied to analyze and design optimization algorithms. 
\subsection{Bregman distance}
\begin{definition}[\cite{lorenz2014the}]\label{Breg}
	Let $f:\mathbb{R}^n\rightarrow\mathbb{R}$ be strongly convex. The Bregman distance $D_f^{\mathbf{x}^*}(\mathbf{x},\mathbf{y})$ between $\mathbf{x},\mathbf{y}\in\mathbb{R}^n$ with respect to $f$ and a subgradient $\mathbf{x}^*\in\partial f(\mathbf{x})$ is defined as 
	$$D_{f}^{\mathbf{x}^*}(\mathbf{x},\mathbf{y}):=f(\mathbf{y})-f(\mathbf{x})-\langle\mathbf{x}^*,\mathbf{y}-\mathbf{x}\rangle.$$
	\end{definition} 
If $f$ is differentiable, then we have $\{\nabla f(\mathbf{x})\}=\partial f(\mathbf{x})$. Note that, when $f(\mathbf{x})=\|\mathbf{x}\|^2$, $D_f(\mathbf{x},\mathbf{y})=\|\mathbf{x}-\mathbf{y}\|^2$, which is the standard Euclidean distance. 

Subsequently, we introduce an important property of the Bregman distance to prove the convergence of SSKM method. It can be immediately derived from the assumption of strong convexity.
\begin{lemma}[\cite{2016Linear}]
	Let $f:\mathbb{R}^n\rightarrow\mathbb{R}$ be $\alpha$-strongly convex. For any $\mathbf{x},\mathbf{y}\in\mathbb{R}^n$, $\mathbf{x}^*\in\partial f(\mathbf{x})$, and $\mathbf{y}^*\in\partial f(\mathbf{y})$, we have 
	$$\frac{\alpha}{2}\|\mathbf{x}-\mathbf{y}\|_2^2\leq D_{f}^{\mathbf{x}^*}(\mathbf{x},\mathbf{y})\leq\langle\mathbf{x}^*-\mathbf{y}^*,\mathbf{x}-\mathbf{y}\rangle\leq\|\mathbf{x}^*-\mathbf{y}^*\|_2\|\mathbf{x}-\mathbf{y}\|_2.$$
	Thus,
	$$D_{f}^{\mathbf{x}^*}(\mathbf{x},\mathbf{y})=0\Leftrightarrow\mathbf{x}=\mathbf{y}.$$
\end{lemma}  

%Proof. The left side can be deduced immediately by the strongly convex property. The right hand side can be obtained by,
%\begin{eqnarray*}
%	&D_f^{\mathbf{x}^*}(\mathbf{x},\mathbf{y})=f(\mathbf{y})-f(\mathbf{x})-<\mathbf{x}^*,\mathbf{y}-\mathbf{x}\geq0&\\
%	&D_f^{\mathbf{y}^*}(\mathbf{y},\mathbf{x})=f(\mathbf{x})-f(\mathbf{y})-<\mathbf{y}^*,\mathbf{x}-\mathbf{y}\geq0&.
%\end{eqnarray*} 

%Next, the definition about the Bregman projection will be given. 
\subsection{Bregman projection}
\begin{definition}[\cite{lorenz2014the}]\label{Bregman Projection}
	Let $f:\mathbb{R}^n\rightarrow\mathbb{R}$ be strongly convex, and $\mathbf{C}\subset\mathbb{R}^n$ be a nonempty closed convex set. The Bregman projection of $\mathbf{x}$ onto $\mathbf{C}$ with respect to $f$ and $\mathbf{x}^*\in\partial f(\mathbf{x})$ is the unique point defined as $\Pi_{\mathbf{C}}^{\mathbf{x}^{*}}(\mathbf{x}) \in \mathbf{C}$ such that 
	$$D_f^{\mathbf{x}^*}(\mathbf{x},\Pi_{\mathbf{C}}^{\mathbf{x}^*}(\mathbf{x}))=\min_{\mathbf{y} \in \mathbf{C}}D_f^{\mathbf{x}^*}(\mathbf{x},\mathbf{y}).$$
	\end{definition} 
Note that, the Bregman projection can be regarded as a generalization of the traditional orthogonal projection. In fact, if $f(\mathbf{x})=\frac{1}{2}\|\mathbf{x}\|^2_2$, then  $\Pi_{\mathbf{C}}^{\mathbf{x}}(\mathbf{x})=\arg\min_{y\in\mathbf{C}}\frac{1}{2}\|\mathbf{x}-\mathbf{y}\|^2_2$. Lemma below characterizes the Bregman projection. 
%For completeness, we give a proof of our own.
\begin{lemma}[\cite{2016Linear}]\label{inequality}
	Let  $f:\mathbb{R}^n\rightarrow\mathbb{R}$ be strongly convex and $\mathbf{C}$ be a nonempty closed convex set. The point $\mathbf{z}\in\mathbf{C}$ is the Bregman projection of $\mathbf{x}$ onto $\mathbf{C}$ with respect to $f$ and $\mathbf{x}^*\in\partial f(\mathbf{x})$ iff there are some $\mathbf{z}^*$ such that one of the following equivalent conditions is satisfied
%	\begin{enumerate}
%			\centering
%		\item $	\langle\mathbf{z}^*-\mathbf{x}^*,\mathbf{y}-\mathbf{z}\rangle\geq 0,\forall\mathbf{y}\in\mathbf{C}$\\
%		\item $D_f^{\mathbf{z}^*}(\mathbf{z},\mathbf{y})\leq D_f^{\mathbf{x}^*}(\mathbf{x},\mathbf{y})-D_f^{\mathbf{x}^*}(\mathbf{x},\mathbf{z})$
%	\end{enumerate}
	\begin{eqnarray*}
	&\langle\mathbf{z}^*-\mathbf{x}^*,\mathbf{y}-\mathbf{z}\rangle\geq 0&,\forall\mathbf{y}\in\mathbf{C}\\
		&D_f^{\mathbf{z}^*}(\mathbf{z},\mathbf{y})\leq D_f^{\mathbf{x}^*}(\mathbf{x},\mathbf{y})-D_f^{\mathbf{x}^*}(\mathbf{x},\mathbf{z}),&
	\end{eqnarray*} 
Then, the point $\mathbf{z}^*$ is called the admissible subgradient for $\mathbf{z}=\Pi_{C}^{\mathbf{x}^*}(\mathbf{x}).$
\end{lemma}
\section{Sparse sampling Kaczmarz-Motzkin method}

In this section, SSKM method will be introduced to solve the augmented Basis pursuit model \cite{Lai,Zhang} 
below,
\begin{eqnarray}
\begin{aligned}
&\min _{\mathbf{x} \in \mathbb{R}^{n}} f(\mathbf{x})=\lambda\|\mathbf{x}\|_{1}+\frac{1}{2}\|\mathbf{x}\|_{2}^{2}\\
&\text {s.t. }\mathbf{A}\mathbf{x}=\mathbf{b},
\end{aligned}
\label{basispursuit1}
\end{eqnarray}
where $\lambda>0$ is some regularizer.
 In this study, we assume $\mathbf{b}\neq\bm{0}$, and it is in the $\text{Range}(\mathbf{A})$. Consequently, the solution of (\ref{basispursuit1}) is unique and nonzero.  
 
 Let $\mathbf{x}_k$ and its admissible subgradient $
 \mathbf{x}^*_k$ be given. The procedure of  SSKM method in each iteration consists of  two steps.
 \begin{enumerate}
 	\item[Step 1.] Choose an index $i_k$ according to the following distribution
 		\begin{eqnarray*}
 			&\tau_k\sim p_{\mathrm{k}}:\left(\begin{array}{c}
 				{[m]} \\
 				\beta_{k}
 			\end{array}\right) \rightarrow[0,1),&\\
 		&i_{k}=\arg \max _{i \in \tau_{k}}( \mathbf{a}_{i}^{\top} \mathbf{x}_{k}-b_{i})^2&,
 	\end{eqnarray*}
 where $\binom{[m]}{\beta_k}$ means sampling $\beta_k$ numbers from the index set $[m]:=\{1,2,\cdots, m\}$.

 	\item[Step 2.] Calculate the Bregman projection of $\mathbf{x}_{k}$ onto the $i_k$-th hyperplane $H(\mathbf{a}_{i_k},b_{i_k})=\{\mathbf{x}:\langle\mathbf{a}_{i_k},\mathbf{x}\rangle=b_{i_k}\}$ denoted by $\Pi^{\mathbf{x}^*_k}_{H(\mathbf{a}_{i_k},b_{i_k})}(\mathbf{x}_{k})$, and calculating its admissible subgradient.
 \end{enumerate}
%The stepsize to iterate $\mathbf{x}_p^*$ can be divided into two categories which are the exact-step and the inexact step. The exact line search corresponding to a Bregman projection onto the hyperplane $H(\mathbf{a}_i,b_i)$.  
\begin{algorithm}[!htb] %算法的开始
	\renewcommand{\algorithmicrequire}{\textbf{Input:}}
	\renewcommand\algorithmicensure {\textbf{Output:} }
	\caption{The Sampling Sparse Kaczmarz-Motzkin method} %算法的标题
	\label{A1} %给算法一个标签，这样方便在文中对算法的引用
	\begin{algorithmic}[1] %这个1 表示每一行都显示数字
		\REQUIRE $\{\mathbf{x}_0=\mathbf{ x}_0^*\in\mathbb{R}^n,\mathbf{A}\in\mathbb{R}^{m\times n},\varepsilon,T\}$ ~~\\ %算法的输入参数：Input
		$\mathbf{x}_0$: the initial point.\\
		$\mathbf{x}_0$: the initial point of intermediate variable.\\
		$\mathbf{ A}$: the measurement matrix, whose rows are normalized.\\
		$\beta_k$: is the random sampling number.\\
		$\varepsilon$:~~the allowed error bound.\\
		$\lambda$: the parameter of the soft-thresholding operator.\\
		$T$:~~the allowed maximum iteration.
		\ENSURE ~~\\ %算法的输出：Output
		$\overline{\mathbf{x}}$: an estimation of the ground truth.\\
		\vskip 4mm
		\hrule
		\vskip 2mm
	\end{algorithmic}
	$\mathbf{Initialization:}$
	\begin{algorithmic}[1]
		\STATE $\mathbf{x}_0,\mathbf{x}_0^*$ are  $\bm{0}$.
	\end{algorithmic}
	$\mathbf{General~step}$
	\begin{algorithmic}[1]
		\STATE
		choose an index $i_k$ from the selection rule.
		\begin{eqnarray*}
			\begin{array}{c}
				\tau_{k}\sim p_k,~\text{and}~
				p_{\mathbf{x}}(\tau_k)=\frac{1}{m},\\
				i_{k}=\arg \max _{i \in \tau_{k}} (\mathbf{a}_{i}^{\top} \mathbf{x}_{k}-b_{i})^2.
			\end{array}
		\end{eqnarray*}
		\STATE
		$\mathbf{x}_{k+1}^{*}=\mathbf{x}_{k}^{*}-t_k \cdot\mathbf{a}_{i_{k}}.$
		\STATE
		Where $t_k=\left\langle\mathbf{a}_{i_{k}}, \mathbf{x}_{k}\right\rangle-\mathbf{b}_{i_{k}},$ which is called the inexact step or $t_{k}=\operatorname{argmin}_{t \in \mathbb{R}} f^{*}\left(\mathbf{x}_{k}^{*}-t \cdot \mathbf{a}_{i_{k}}\right)+t \cdot\mathbf{b}_{i_{k}}$ which is called the exact step.
		\STATE
		$\mathbf{x}_{k+1}=S_{\lambda}\left(\mathbf{x}_{k+1}^{*}\right),$ where $S_{\lambda}(\cdot)$ is the soft-thresholding operator which is defined as $$S_{\lambda}(x_i)=\max(|x_i|-\lambda,0)\cdot \textrm{sign}(x_i),i=1,\cdots,n.$$
		\IF{$\big|\big|\mathbf{ A}\mathbf{ x}_{k+1}-\mathbf{b}\big|\big|_2\leq\varepsilon$ \textbf{or} $k=T$}
		\STATE $\overline{\mathbf{x}}=\mathbf{x}_{k+1}$.\\
		\ENDIF
	\end{algorithmic}
\end{algorithm}

In the following, we will present some technical details that help us understand SSKM, along with some preliminary theoretical results, which will be used for convergence analysis.
\subsection{Sampling rule} 
The strategy to choose which $H(\mathbf{a}_i,b_i)$ to be projected onto is based on the rule of Sampling Kaczmarz-Motzkin method, which picks up the most violated item from one of the subsets with $\beta$ rows from $\mathbf{A}$. On the contrary to the randomized Kaczmarz  and Randomized sparse Kazmarz methods, which pick up rows with a fixed probability,  the probability that SSKM method utilized is flexible, and it is defined as 
\begin{eqnarray}\label{probability}
p_{\mathbf{x}}(\tau_k):=\frac{\|\mathbf{a}_{t(\tau_k,\mathbf{ x})}\|^2_2}{\sum_{\tau_k\in\binom{[m]}{\beta_k}}\|\mathbf{a}_{t(\tau_k,\mathbf{ x})}\|^2_2},
\end{eqnarray}
where $t(\tau_k,\mathbf{x})=\arg\max_{i\in\tau_k}(\mathbf{a}_i^T\mathbf{ x}-b_i)^2$. From (\ref{probability}), we can find the probability to choose the sub-row of $\mathbf{A}$ is not uniform, it depends on the norm of  $\mathbf{a}_{t(\tau_k, \mathbf{x})}$ which has the largest error among $(\mathbf{a}_i^T\mathbf{ x}-b_i)^2,i\in\tau_k$. If $\beta_k=1$, (\ref{probability}) is equivalent to the randomized Kaczmarz method.  If $\beta_k=m$, it is equivalent to picking up the most violated item from the whole rows of $\mathbf{ A}$. 

At the first glimpse, calculating (\ref{probability}) demands a large burden of computational cost. However, if $\mathbf{A}$ is normalized, (\ref{probability}) is equal to choosing the most violated item from random $\binom{[m]}{\beta_k}$ rows of $\mathbf{A}$. So, there is no need to find out all the subsets with $\beta_k$ rows of $\mathbf{A}$. As a result, the whole computational cost is low. 
 \subsection{Bregman projection procedure}
 The core of the second part is to calculate the Bregman projection. Lemma \ref{t1} below demonstrates how to calculate the Bregman projection of a given point. For completeness, we give a simple theoretical proof here.
 \begin{lemma}[\cite{lorenz2014the}]\label{t1}
 	Let $f:\mathbb{R}^n\rightarrow\mathbb{R}$ be $\alpha$-strongly convex, $\mathbf{A}\in\mathbb{R}^{m\times n}$,$\mathbf{b}\in\mathbb{R}^m$. Then, the Bregman projection of $\mathbf{x}\in\mathbb{R}^n$ onto the hyperplane $H(\mathbf{a}_i,b_i)$ with $\mathbf{a}_i\neq\bm{0}$ is 
 	\begin{eqnarray}
 	\mathbf{z}:=\Pi_{H(\mathbf{a}_i,b_i)}^{\mathbf{x}^*}(\mathbf{x})=\nabla f^{*}(\mathbf{x}^*-\hat{t}\mathbf{a}_i),
 	\end{eqnarray}
 	where $\hat{t}\in\mathbb{R}$, which is one of the solutions of 
 	$$\min_{t\in\mathbb{R}}f^*(\mathbf{x}^*-t\mathbf{a}_i)+tb_i.$$
 	Moreover, $\mathbf{z}^*:=\mathbf{x}^*-\hat{t}\mathbf{a}_i$ is an admissible subgradient for $\mathbf{z}$ and for any $\mathbf{y}\in H(\mathbf{a}_i,b_i)$, we have
 	\begin{eqnarray}\label{BIN}
 	D_f^{\mathbf{z}^*}(\mathbf{z},\mathbf{y})\leq D_f^{\mathbf{x}^*}(\mathbf{x},\mathbf{y})-\frac{\alpha}{2}\frac{(\langle\mathbf{a}_i,\mathbf{x}\rangle-b_i)^2}{\|\mathbf{a}_i\|^2_2}.
 	\end{eqnarray} 
 \end{lemma}
 \begin{proof}
 	Recall the definition of the hyperplane $H(\mathbf{ a}_i,b_i)=\{\mathbf{x}:\mathbf{a}_i^{T}\mathbf{x}=b_i\}$. If $\mathbf{z}:=\Pi^{\mathbf{x}^*}_{H(\mathbf{a}_i,b_i)}(\mathbf{x})$, then $\mathbf{a}_i^{T}\mathbf{z}=b_i$. As a result, $H(\mathbf{ a}_i,b_i)-\{\mathbf{z}\}=\{\mathbf{y}:\mathbf{a}_i^{T}\mathbf{y}=0\}$.
 	Thus, the normal cone $N_{H(\mathbf{a}_i,b_i)}(\mathbf{z})=\{t\mathbf{a}_i,t\in\mathbb{R}\}$. 
 	Because
 	\begin{eqnarray*}
 	\mathbf{z}&=&\min_{\mathbf{y}\in\mathbf{C}}D_f^{\mathbf{x}^*}(\mathbf{x},\mathbf{y})=\min_{\mathbf{y}\in\mathbf{C}}f(\mathbf{y})-\langle\mathbf{x}^*,\mathbf{y}-\mathbf{x}\rangle\\\nonumber&=&\min_{\mathbf{y}\in\mathbb{R}^n} f(\mathbf{y})-\langle\mathbf{x}^*,\mathbf{y}-\mathbf{x}\rangle+\delta_{\mathbf{C}}(\mathbf{y}),
 \end{eqnarray*} 
where $\delta_{\mathbf{C}}(\cdot)$ is the indicator function. Thus, we can infer that $(\mathbf{x}^*-\mathbf{z}^*)\in N_{\mathbf{C}}(\mathbf{x})$ by the first order optimality condition. As a result, there exists
 	$\hat{t}\in\mathbb{R}$, such that $\mathbf{x}^*-\hat{t}\mathbf{a}_i\in\partial f(\mathbf{z})$. Then $\nabla f^*(\mathbf{x}^*-\hat{t}\mathbf{a}_i)=\mathbf{z}$, which is the Bregman projection of $\mathbf{x}$. Remaining is to calculate $\hat{t}$, 	which is called the exact step. Note that $\mathbf{a}_i^T\nabla f^*(\mathbf
 	{x}^*-\hat{t}\mathbf{a}_i)=b_i.$ By the first order optimality condition, we can formulate an optimization problem below
 	$$\hat{t}\in \arg\min_{t\in\mathbb{R}}f^*(\mathbf{x}^*-t\mathbf{a}_i)+tb_i.$$
 	
 \noindent 
 Last is to prove the inequality. Recalling Lemma \ref{inequality}, we have
 	$$D_f^{\mathbf{z}^*}(\mathbf{z},\mathbf{y})\leq D_f^{\mathbf{x}^*}(\mathbf{x},\mathbf{y})-D_f^{\mathbf{x}^*}(\mathbf{x},\mathbf{z}).$$
 	Then, by the strong convex property of $f$, we derive that
 	$$D_f^{\mathbf{x}^*}(\mathbf{x},\mathbf{z})\geq\frac{\alpha}{2}\|\mathbf{x}-\mathbf{z}\|^2\geq\frac{\alpha}{2}\|\mathbf{x}-\mathbb{P}_{H(\mathbf{a}_i,b_i)}(\mathbf{x})\|^2=\frac{\alpha}{2}\frac{(\langle\mathbf{a}_i,\mathbf{x}
 		\rangle-b_i)^2}{\|\mathbf{a}_i\|^2_2},$$
 	where $\mathbb{P}_{H(\mathbf{a}_i,b_i)}(\mathbf{x})$ is the standard orthogonal projection of $\mathbf{x}$ onto $H(\mathbf{a}_i,b_i)$. Then, we complete the proof.
 \end{proof}

\section{Proof of linear convergence of the SSKM method}

%The stepsize to iterate $\mathbf{x}_p^*$ can be divided into two categories which are the exact-step and the inexact step. The exact line search corresponding to a Bregman projection onto the hyperplane $H(\mathbf{a}_i,b_i)$.  

First, we characterize the error bound between $D_f^{\mathbf{x}^*}(\mathbf{x},\hat{\mathbf{x}})$ and $\|\mathbf{ A}\mathbf{x}-\mathbf{ b}\|^2$.
\begin{lemma}[ \cite{2016Linear}]\label{errorbound}
	Let $\tilde{\sigma}_{\min}(\mathbf{ A})$ and $|\hat{\mathbf{x}}|_{\min}$ be defined as before. When $\lambda>0$, then for any $\mathbf{x}\in\mathbb{R}^n$ with $\partial f(\mathbf{x})\cap\mathcal{R}(\mathbf{ A}^{T})\neq\emptyset$ and for all $\mathbf{x}^*=\mathbf{A}^{T}\mathbf{y}\in\partial f(\mathbf{x})\cap\mathcal{R}(\mathbf{ A}^{T})$, 
	$\mathbf{y}\in\mathbb{R}^m$, we have
	\begin{eqnarray}\label{error} 	D_f^{\mathbf{x}^*}(\mathbf{x},\hat{\mathbf{x}})\leq\frac{1}{\tilde{\sigma}_{\min}^2(\mathbf{A})}\cdot\frac{|\hat{\mathbf{ x}}|_{\min}+2\lambda}{|\hat{\mathbf{ x}}|_{\min}}\|\mathbf{ A}\mathbf{x}-\mathbf{b}\|^2_2.
	\end{eqnarray}
\end{lemma}

%\begin{proof}
%	Let $\mathbf{W}$ be the minimizers of the function  $F(\bm{\omega}):=f^*(\mathbf{x}-\mathbf{ A}^{T}\bm{\omega})+\mathbf{ b}^{T}\bm{\omega}$. Then for $\forall\tilde{\omega}\in\mathbf{W}$,Recall that $\hat{\mathbf{x}}=\nabla f^*(\mathbf{x}^*-\mathbf{ A}^{T}\tilde{\bm{\omega}})$, thus $\hat{\mathbf{ x}}^*=\mathbf{x}^*-\mathbf{ A}^{T}\tilde{\bm{\omega}}$,
%	then we have 
%	\begin{eqnarray}\label{errorbound2}
%	D_f^{\mathbf{x}^*}(\mathbf{x},\hat{\mathbf{x}})&\leq&<\hat{\mathbf{x}}^*-\mathbf{x}^*,\hat{\mathbf{x}}-\mathbf{x}>=<\tilde{\bm{\omega}},\mathbf{ A}\mathbf{x}-\mathbf{b}
%	>\nonumber\\
%	&\leq&\|\tilde{\bm{\omega}}\|\|\mathbf{A}\mathbf{x}-\mathbf{b}\|
%	\end{eqnarray}
%	$\forall\mathbf{w}\in\mathbb{R}^n$, because $F(\omega)$ is a quadratic piece-wise linear function, we can obtain the error bound below,
%	\begin{eqnarray}
%	\min_{\bm{\omega} \in\mathbf{ W}}\|\mathbf{w}-\bm{\omega}\|=dist(\mathbf{w},\mathbf{W})\leq\mu\|\nabla F(\mathbf{w})\|
%	\end{eqnarray} 
%	where $\mu$ is the error bound parameter, we define it here as the minimum value which satisfies the inequality below. As a result, if let $\mathbf{w}=\bm{0}$, we can have,
%	\begin{eqnarray}
%	\min_{\mathbf{w}\in\mathbf{W}}\|\bm{\omega}\|\leq\mu\|\mathbf
%	{A}\nabla f^*(\mathbf{x}^*)-\mathbf{b}\|=\mu\|\mathbf{A}\mathbf{x}-\mathbf{b}\|
%		\end{eqnarray}
%    Combining with the (\ref{errorbound2}), the inequality can be obtained as below,
%    \begin{eqnarray}
%    D_f^*(\mathbf{x},\hat{\mathbf{x}})\leq\mu\|\mathbf{A}\mathbf{x}-\mathbf{b}\|^2.
%    \end{eqnarray}
%\end{proof}
Remark: when $\lambda=0$, and $\mathbf{A}$ is a full column rank matrix, by the strong convexity of $f$, we can immediately obtain that
\begin{eqnarray}\label{eq}
D_f^{\mathbf{x}^*}(\mathbf{x},\hat{\mathbf{x}})\leq\frac{1}{2\sigma_{\min}^2(\mathbf{A})}\|\mathbf{ A}\mathbf{x}-\mathbf{b}\|_2^2.
\end{eqnarray} 
In the following, we present our results for the noiseless and noisy cases respectively.
\begin{theorem}[Noiseless case]\label{main}
	Let $$\gamma_k:=\frac{\sum_{\tau_k\in\binom{[m]}{\beta_k}}\|\mathbf{A}_{\tau_k}\mathbf{x}_{k}-\mathbf{b}_{\tau_k}\|^2_2}{\sum_{\tau_k\in\binom{[m]}{\beta_k}}\|\mathbf{A}_{\tau_k}\mathbf{x}_{k}-\mathbf{b}_{\tau_k}\|^2_{\infty}}\leq\beta_k,$$
	and
	$$q_k:= \left\{
	\begin{aligned}
	&(1-\frac{\beta_k\tilde{\sigma}_{\min}^2(\mathbf{A})}{2\gamma_km}\cdot\frac{|\hat{\mathbf{ x}}|_{\min}}{|\hat{\mathbf{ x}}|_{\min}+2\lambda}),\lambda>0&\\
	&(1-\frac{\beta_k\sigma_{\min}^2(\mathbf{A})}{\gamma_km}),\lambda=0&
	\end{aligned}
	\right..$$
	
	\noindent
	The sequence $\{\mathbf{x}_k\}$ generated by the SSKM method in Algorithm \ref{A1} converges linearly in expectation to the unique solution $\hat{\mathbf{x}}$ of (\ref{basispursuit1}) in the sense that
	\begin{eqnarray}\label{a1}
		\mathbb{E}[D_f^{\mathbf{x}^*_{k+1}}(\mathbf{x}_{k+1},\hat{\mathbf{x}})]\leq q_k\mathbb{E}[ D_f^{\mathbf{x}_k^*}(\mathbf{x}_k,\hat{\mathbf{x}})].
	\end{eqnarray}	
	\noindent Furthermore, we have
	\begin{eqnarray}\label{A2}
		\mathbb{E}[\|\mathbf{ x}_{k+1}-\hat{\mathbf{ x}}\|_2]\leq\prod_{i=0}^{k} q_k^{\frac{1}{2}}\sqrt{2\lambda\|\hat{\mathbf{x}}\|^2_1+\|\hat{\mathbf{x}}\|_2^2}.
	\end{eqnarray}
\end{theorem}
Remark:{ \textit{The convergence rate of SSKM method depends on the contraction factors $q_k$. Note that $1\leq\gamma_k\leq\beta_k$. When $(\mathbf{a}^{T}_i\mathbf{x}_k-b_i)^2, i=1,\cdots,m$ are equal,  the upper bound of $\gamma_k$ by $\beta_k$ can be achieved. In that case, SSKM method obtains the slowest convergence rate. On the contrary, the lower bound $\gamma_k=1$ can be achieved when only one of the residuals $|\mathbf{a}^{T}_i\mathbf{x}_k-b_i|$ is nonzero. Then, SSKM method converges as fast as 
		$$q_k= \left\{
		\begin{aligned}
		&(1-\frac{\tilde{\sigma}_{\min}^2(\mathbf{A})}{2}\cdot\frac{|\hat{\mathbf{ x}}|_{\min}}{|\hat{\mathbf{ x}}|_{\min}+2\lambda}),\lambda>0&\\
		&(1-\sigma^2_{\min}(\mathbf{A})),\lambda=0&
		\end{aligned}
		\right..$$ }}
%	\begin{eqnarray*}
%		\frac{\sum_{\tau_{k} \in\binom{[m]}{\beta_{k}}}\left\|\mathbf{A}_{\tau_{k}} \mathbf{x}_{k}-\mathbf{b}_{\tau_{k}}\right\|^{2}}{\sum_{\tau_{k} \in\binom{[m]}{\beta_{k}}}\left\|\mathbf{A}_{\tau_{k}} \mathbf{x}_{k}-\mathbf{b}_{\tau_{k}}\right\|^{2}_{\infty}}&=& \frac{\frac{\beta_{k}}{m}\binom{[m]}{ \beta_{k}}\left\|\mathbf{A}\left(\mathbf{x}_{k}-\mathbf{x}^{*}\right)\right\|^{2}}{\sum_{\tau_k\in\binom{[m]}{\beta_k}}\mid\mathbf{a}_{t\left(\tau_{k}, \mathbf{x}_{k}\right)}^{\top}\left(\mathbf{x}_{k}-\mathbf{x}^{*}\right) \mid}\\
%		&\geq&\frac{\frac{\beta_k}{m}\binom{[m]}{\beta_k}\sigma^2_{min}(\mathbf{A})\|\mathbf{x}_k-\mathbf{x}^*\|^2}{\sum_{\tau_k\in\binom{[m]}{\beta_k}}\left\|\mathbf{a}_{t\left(\tau_{k}, \mathbf{x}_{k}\right)}\right\|^{2}\left\|\mathbf{x}_{k}-\mathbf{x}^{*}\right\|^{2}}\\
%		&=&\frac{\beta_{k}\binom{[m]}{\beta_k} \sigma_{\min }^{2}(\mathbf{A})}{m \sum_{\binom{[m]}{\beta_k}} \| \mathbf{a}_{t\left(\tau, \mathbf{x}_{k}\right)} \|^{2}}
%	\end{eqnarray*}	
%	then,we can get the bound $\gamma_k\geq\frac{\beta_k\binom{m}{\beta_k}\sigma^2_{min}(\mathbf{ A})}{m\sum_{\tau\in\binom{[m]}{\beta_k}}^{}\|\mathbf{a}_{t(\tau,\mathbf{ x}_k)}\|^2}$. And $\tilde{\sigma}_{min}(\mathbf{ A})\geq\sigma_{min}(\mathbf{ A})$. As a result, we can have a faster convergence rate than SRK method. Notice that, when $\lambda=0$, combining with Lemma \ref{inequality}, we can recover the same convergence rate with the Sampling Motzkin Kaczmarz method\cite{haddock2019greed}. 
\textit{The relationship between $\gamma_k$ and $\beta_k$ can be refined for different $\mathbf{A}$. In \cite{haddock2019greed}, it demonstrates that when $\mathbf{A}$ are drawn i.i.d from a standard Gaussian distribution, $\gamma_k=\mathcal{O}(n\beta_k/log(\beta_k))$.}

Now, let us finish the proof of Theorem \ref{main}.

\begin{proof}
By using (\ref{BIN}) with $\mathbf{x}_{k+1}=\mathbf{z}$, $\mathbf{x}_k=\mathbf{x}$ and $\mathbf{y}=\hat{\mathbf{x}}$ in Lemma \ref{t1}, inequality can be reformulated as below
\begin{eqnarray}\label{backbone}
	D_{f} ^{\mathbf{x}_{k+1}^*}(\mathbf{x}_{k+1},\hat{\mathbf{x}})
	&\leq& D_{f} ^{\mathbf{x}_{k}^*}(\mathbf{x}_{k},\hat{\mathbf{x}})-\frac{1}{2}\cdot\frac{(\mathbf{a}_{i(\tau_k,\mathbf{x}_k)}\mathbf{x}_k-\mathbf{ b}_{i_k})^2}{\|\mathbf{a}_{i(\tau_k,\mathbf{x}_k)}\|^2_2} \\
	&=& D_{f} ^{\mathbf{x}_{k}^*}(\mathbf{x}_{k},\hat{\mathbf{x}})-\frac{1}{2}\cdot\frac{\|\mathbf{A}_{\tau_k}\mathbf{x}_k-\mathbf{ b}_{\tau_k}\|^2_{\infty}}{\|\mathbf{a}_{i(\tau_k,\mathbf{x}_k)}\|^2_2}\nonumber. 
\end{eqnarray}
At the same time, (\ref{backbone}) is also held for the inexact step by Theorem 2.8 in \cite{lorenz2014the}. Using the sampling rule of the Kaczmarz-Motzkin method, and treating $\tau_k$ as the random variable, we derive that
\begin{eqnarray*}
&\mathbb{E}[D_{f}^{\hat{\mathbf{x}_{k+1}}}&(\mathbf{x}_{k+1},\hat{\mathbf{x}})|\tau_{k-1},\cdots,\tau_1,\tau_0]\\\nonumber
&\leq& D_f^{\mathbf{x}_k^*}(\mathbf{x}_k,\hat{\mathbf{x}})-\mathbb{E}_{\tau_k}(\frac{1}{2}\cdot\frac{\|\mathbf{A}_{\tau_k}\mathbf{x}_k-\mathbf{ b}_{\tau_k}\|^2_{\infty}}{\|\mathbf{a}_{i(\tau_k,\mathbf{x}_k)}\|^2_2})\\
&=& D_{f} ^{\mathbf{x}_{k}^*}(\mathbf{x}_{k},\hat{\mathbf{x}})-\sum_{\tau_k\in\binom{[m]}{\beta_k}}^{}p_{\mathbf{x}_{k}}(\tau_k)\frac{1}{2}\frac{\|\mathbf{A}_{\tau_k}\mathbf{x}_k-\mathbf{ b}_{\tau_k}\|^2_{\infty}}{\|\mathbf{a}_{i(\tau_k,\mathbf{x}_k)}\|^2_2}\\
 &=&D_{f} ^{\mathbf{x}_{k}^*}(\mathbf{x}_{k},\hat{\mathbf{x}})-\sum_{\tau_k\in\binom{[m]}{\beta_k}}^{}\frac{\|\mathbf{ a}_{i(\tau_k,\mathbf{x}_k)}\|^2_2}{\sum_{\pi\in\binom{[m]}{\beta_k}}^{}\|\mathbf{ a}_{i(\pi,\mathbf{x}_k)}\|^2_2}\frac{1}{2}\frac{\|\mathbf{A}_{\tau_k}\mathbf{x}_k-\mathbf{ b}_{\tau_k}\|^2_{\infty}}{\|\mathbf{a}_{i(\tau_k,\mathbf{x}_k)}\|^2_2}\\
 &=&
 D_{f} ^{\mathbf{x}_{k}^*}(\mathbf{x}_{k},\hat{\mathbf{x}})-\sum_{\tau_k\in\binom{[m]}{\beta_k}}^{}\frac{1}{2}\frac{\|\mathbf{A}_{\tau_k}\mathbf{x}_k-\mathbf{ b}_{\tau_k}\|^2_{\infty}}{\sum_{\pi\in\binom{[m]}{\beta_k}}^{}\|\mathbf{ a}_{i(\pi,\mathbf{x}_k)}\|^2_2}\\
 &=&
  D_{f} ^{\mathbf{x}_{k}^*}(\mathbf{x}_{k},\hat{\mathbf{x}})-\frac{1}{2}\frac{\binom{m}{\beta_k}\beta_k\|\mathbf{A}\mathbf{x}_k-\mathbf{ b}\|^2_2}{\gamma_k\sum_{\pi\in\binom{[m]}{\beta_k}}^{}\|\mathbf{ a}_{i(\pi,\mathbf{x}_k)}\|^2_2}\\
  &\leq&
(1-\frac{\beta_k\tilde{\sigma}_{\min}^2(\mathbf{A})}{2\gamma_km}\cdot\frac{|\hat{\mathbf{ x}}|_{\min}}{|\hat{\mathbf{ x}}|_{\min}+2\lambda}) D_{f} ^{\mathbf{x}_{k}^*}(\mathbf{x}_{k},\hat{\mathbf{x}}),
\end{eqnarray*}
where the last equality can be derived from (\ref{error}). Now considering all indexes $\tau_0,\cdots,\tau_k$ as random variables with values in $\{1,\cdots, m\}$,
and taking the full expectation on both sides,  we can finish the proof of (\ref{a1}).

Note that
$$D_f^{\mathbf{x}^*_{k+1}}(\mathbf{x}_{k+1},\hat{\mathbf{x}})\geq\frac{1}{2}\|\mathbf{x}^{k+1}-\hat{\mathbf{x}}\|_2^2,$$
\noindent and
$$D_f^{\mathbf{x}^*_{0}}(\mathbf{x}_{0},\hat{\mathbf{x}})=\frac{1}{2}\|\hat{\mathbf{x}}\|^2_2+\lambda\|\hat{\mathbf{x}}\|_1.$$
Inductively, we can derive (\ref{A2}).
\end{proof}

Next, we turn to the noisy case by following the idea of proof in\cite{2016Linear}.
\begin{theorem}[Noisy case]\label{NoiseT}
	Assume that a noisy observed data $\mathbf{b}^{\delta}\in\mathbb{R}^m$ with $\|\mathbf{b}^{\delta}-\mathbf{b}\|_2\leq\delta$ is given, where $\mathbf{b}=\mathbf{ A}\hat{\mathbf{ x}}$. If the sequence $\{\mathbf{x}_k\}$ generated by SSKM method in Algorithm \ref{A1} are computed by $\mathbf{b}^{\delta}$. Then, with the same contraction factor $q_i$ as in the noiseless case,
	for the inexact step, we can have,
	$$\mathbb{E}[\|\mathbf{x}_{k+1}-\hat{\mathbf{x}}\|_2]\leq \sqrt{\prod_{i=0}^{k}q_{i} \cdot\left(2\lambda\|\hat{\mathbf{x}}\|_{1}+\|\hat{\mathbf{x}}\|_{2}^{2}\right)}+\sqrt{ \frac{\sum_{i=0}^{k}q_i\delta^2}{2}},$$
	and for the exact step,
	$$\mathbb{E}[\|\mathbf{x}_{k+1}-\hat{\mathbf{x}}\|_2]\leq \sqrt{\prod_{i=0}^{k}q_{i} \cdot\left(2\lambda\|\hat{\mathbf{x}}\|_{1}+\|\hat{\mathbf{x}}\|_{2}^{2}\right)}+\sqrt{\sum_{i=0}^{k}q_i\delta^2 \cdot \frac{1+4\lambda\|\mathbf{ A}\|_{1,2}}{2}}.$$
\end{theorem}
\begin{proof}
Define $\mathbf{x}_{k}^{\delta}:=\hat{\mathbf{x}}+\frac{b_{i_{k}}^{\delta}-b_{i_{k}}}{\left\|\mathbf{a}_{i_{k}}\right\|_{2}^{2}} \cdot\mathbf{a}_{i_{k}}$, then we can find $\mathbf{x}^{\delta}_k\in H\left(\mathbf{a}_{i_{k}}, b_{i_{k}}^{\delta}\right)$. By Lemma \ref{t1}, we have the inequality below,
\begin{eqnarray}\label{ne1}
D_{f}^{\mathbf{x}_{k+1}^{*}}\left(\mathbf{x}_{k+1}, \mathbf{x}_{k}^{\delta}\right) \leq D_{f}^{\mathbf{x}_{k}^{*}}\left(\mathbf{x}_{k}, \mathbf{x}_{k}^{\delta}\right)-\frac{1}{2} \cdot \frac{\left(\left\langle \mathbf{a}_{i_{k}}, \mathbf{x}_{k}\right\rangle-b_{i_{k}}^{\delta}\right)^{2}}{\left\|\mathbf{a}_{i_{k}}\right\|_{2}^{2}}.
\end{eqnarray}

\noindent
Unfolding the expression of $D_{f}^{\mathbf{x}_{k+1}^{*}}\left(\mathbf{x}_{k+1}, \mathbf{x}_{k}^{\delta}\right)$ and $D_{f}^{\mathbf{x}_{k}^{*}}\left(\mathbf{x}_{k}, \mathbf{x}_{k}^{\delta}\right)$, and plugging $f(\hat{\mathbf{x}})$ into both sides of (\ref{ne1}),  we obtain the inequality
\begin{eqnarray}\label{n1}
D_{f}^{\mathbf{x}_{k+1}^{*}}\left(\mathbf{x}_{k+1}, \hat{\mathbf{x}}\right) &\leq& D_{f}^{\mathbf{x}_{k}^{*}}\left(\mathbf{x}_{k}, \hat{\mathbf{x}}\right)-\frac{1}{2} \cdot \frac{\left(\left\langle\mathbf{a}_{i_{k}}, \mathbf{x}_{k}\right\rangle-b_{i_{k}}^{\delta}\right)^{2}}{\left\|\mathbf{a}_{i_{k}}\right\|_{2}^{2}}\nonumber\\
&~&+\left\langle\mathbf{x}_{k+1}^{*}-\mathbf{x}_{k}^{*}, \mathbf{x}_{k}^{\delta}-\hat{\mathbf{x}}\right\rangle.
\end{eqnarray}
In the remaining, we prove Theorem \ref{NoiseT} from the cases of exact step and the in-exact step.

$\textbf{Inexact step.}$ For the inexact step, observe that $\mathbf{x}_{k+1}^{*}-\mathbf{x}_{k}^{*}=-\frac{\left\langle \mathbf{a}_{i_{k}}, \mathbf{x}_{k}\right\rangle-b_{i_{k}}^{\delta}}{\left\|\mathbf{a}_{i_{k}}\right\|_{2}^{2}} \cdot \mathbf{a}_{i_{k}}$. Then, we bound the inner product in (\ref{n1}) as below,
\begin{eqnarray}\label{s1}
\begin{aligned}
\left\langle \mathbf{x}_{k+1}^{*}-\mathbf{x}_{k}^{*}, \mathbf{x}_{k}^{\delta}-\hat{\mathbf{x}}\right\rangle &=\frac{b_{i_{k}}^{\delta}-b_{i_{k}}}{\left\|\mathbf{a}_{i_{k}}\right\|_{2}^2} \cdot\left\langle \mathbf{x}_{k+1}^{*}-\mathbf{x}_{k}^{*}, \mathbf{a}_{i_{k}}\right\rangle \\
&=\frac{\left(b_{i_{k}}^{\delta}-b_{i_{k}}\right)^{2}}{\left\|\mathbf{a}_{i_{k}}\right\|_{2}^{2}}-\frac{\left(b_{i_{k}}^{\delta}-b_{i_{k}}\right) \cdot\left(\left\langle\mathbf{a}_{i_{k}}, \mathbf{x}_{k}\right\rangle-b_{i_{k}}\right)}{\left\|\mathbf{a}_{i_{k}}\right\|_{2}^{2}} .
\end{aligned}
\end{eqnarray}

\noindent By reformulating
\begin{eqnarray}\label{s2}
-\frac{1}{2} \cdot \frac{\left(\left\langle\mathbf{a}_{i_{k}}, \mathbf{x}_{k}\right\rangle-b_{i_{k}}^{\delta}\right)^{2}}{\left\|\mathbf{a}_{i_{k}}\right\|_{2}^{2}}&=&-\frac{1}{2} \cdot \frac{\left(\left\langle\mathbf{a}_{i_{k}}, \mathbf{x}_{k}\right\rangle-b_{i_{k}}\right)^{2}}{\left\|\mathbf{a}_{i_{k}}\right\|_{2}^{2}}\nonumber\\
&~&+\frac{\left(b_{i_{k}}^{\delta}-b_{i_{k}}\right) \cdot\left(\left\langle \mathbf{a}_{i_{k}}, \mathbf{x}_{k}\right\rangle-b_{i_{k}}\right)}{\left\|\mathbf{a}_{i_{k}}\right\|_{2}^{2}}-\frac{1}{2} \cdot \frac{\left(b_{i_{k}}^{\delta}-b_{i_{k}}\right)^{2}}{\left\|\mathbf{a}_{i_{k}}\right\|_{2}^{2}}
\end{eqnarray}
and combining (\ref{s1}) and (\ref{s2}) into (\ref{n1}), we have
\begin{eqnarray}
D_{f}^{\mathbf{x}_{k+1}^{*}}\left(\mathbf{x}_{k+1}, \hat{\mathbf{x}}\right) \leq D_{f}^{\mathbf{x}_{k}^{*}}\left(\mathbf{x}_{k}, \hat{\mathbf{x}}\right)-\frac{1}{2} \cdot \frac{\left(\left\langle \mathbf{a}_{i_{k}}, \mathbf{x}_{k}\right\rangle-b_{i_{k}}\right)^{2}}{\left\|\mathbf{a}_{i_{k}}\right\|_{2}^{2}}+\frac{1}{2} \cdot \frac{\left(b_{i_{k}}^{\delta}-b_{i_{k}}\right)^{2}}{\left\|\mathbf{a}_{i_{k}}\right\|_{2}^{2}}.
\end{eqnarray}
Similar to the proof in the noiseless situation, we have
\begin{eqnarray}
\mathbb{E}\left[D_{f}^{\mathbf{x}_{k+1}^{*}}\left(\mathbf{x}_{k+1}, \hat{\mathbf{x}}\right)\right] \leq q_k \cdot \mathbb{E}\left[D_{f}^{\mathbf{x}_{k}^{*}}\left(\mathbf{x}_{k}, \hat{\mathbf{x}}\right)\right]+\frac{1}{2}\left\|\mathbf{b}^{\delta}-\mathbf{b}\right\|_{\infty}^{2},
\end{eqnarray}
where
$\sum_{\tau_k\in\binom{[m]}{\beta_k}}^{}\frac{1}{2}\frac{\left(b_{i(\tau_k, \mathbf{x}_k)}^{\delta}-b_{i(\tau_k,\mathbf{x}_k)}\right)^{2}}{\sum_{\pi\in\binom{[m]}{\beta_k}}^{}\|\mathbf{ a}_{i(\pi,\mathbf{x}_k)}\|^2}\leq\frac{1}{2} \cdot \frac{\binom{m}{\beta_k}\left\|\mathbf{b}^{\delta}-\mathbf{b}\right\|_{\infty}^{2}}{\binom{m}{\beta_k}}=\frac{1}{2}\left\|\mathbf{b}^{\delta}-\mathbf{b}\right\|_{\infty}^{2}.$

\noindent
Inductively, we get 
\begin{eqnarray*}
\mathbb{E}\left[D_{f}^{\mathbf{x}_{k+1}^{*}}\left(\mathbf{x}_{k}, \hat{\mathbf{x}}\right)\right] \leq \prod_{i=0}^{k}q_{i} \cdot\left(2\lambda\|\hat{\mathbf{x}}\|_{1}+\|\hat{\mathbf{x}}\|_{2}^{2}\right)+ \frac{\sum_{i=0}^{k}q_i}{2} \cdot \left\|\mathbf{b}^{\delta}-\mathbf{b}\right\|_{\infty}^{2}.
\end{eqnarray*}
Finally,
\begin{eqnarray*}
\mathbb{E}[\|\mathbf{x}_{k+1}-\hat{\mathbf{x}}\|_2]\leq \sqrt{\prod_{i=0}^{k}q_{i} \cdot\left(2\lambda\|\hat{\mathbf{x}}\|_{1}+\|\hat{\mathbf{x}}\|_{2}^{2}\right)}+\sqrt{ \frac{\sum_{i=0}^{k}q_i}{2} \cdot\left\|\mathbf{b}^{\delta}-\mathbf{b}\right\|_{\infty}^{2}}.
\end{eqnarray*}

$\textbf{Exact step.}$ The idea to prove Theorem \ref{NoiseT} in the exact step case is similar. But for the exact step, $\mathbf{x}_k^*=\mathbf{x}_k+\lambda\mathbf{s}_k$, where $\|\mathbf{s}_k\|_{\infty}\leq 1$. Noting that $\langle\mathbf{x}_{k+1},\mathbf{a}_{i_k}\rangle=b_{i_k}$, thus we derive that
\begin{eqnarray*}
	\left\langle \mathbf{x}_{k+1}^{*}-\mathbf{x}_{k}^{*}, \mathbf{x}_{k}^{\delta}-\hat{\mathbf{x}}\right\rangle &=&\frac{b_{i_{k}}^{\delta}-b_{i_{k}}}{\left\|\mathbf{a}_{i_{k}}\right\|_{2}^{2}} \cdot\left(\left\langle \mathbf{x}_{k+1}-\mathbf{x}_{k}, \mathbf{a}_{i_{k}}\right\rangle+\left\langle \mathbf{s}_{k+1}-\mathbf{s}_{k}, \mathbf{a}_{i_{k}}\right\rangle\right) \\
	&=&\frac{b_{i_{k}}^{\delta}-b_{i_{k}}}{\left\|\mathbf{a}_{i_{k}}\right\|_{2}^{2}} \cdot \left(b_{i_k}^{\delta}-b_{i_k}+b_{i_k}-\left\langle \mathbf{x}_{k}, \mathbf{a}_{i_{k}}\right\rangle+\left\langle \mathbf{s}_{k+1}-\mathbf{s}_{k}, \mathbf{a}_{i_{k}}\right\rangle\right)\\
	& \leq& \frac{\left(b_{i_{k}}^{\delta}-b_{i_{k}}\right)^{2}}{\left\|\mathbf{a}_{i_{k}}\right\|_{2}^{2}}-\frac{\left(b_{i_{k}}^{\delta}-b_{i_{k}}\right) \cdot\left(\left\langle \mathbf{a}_{i_{k}}, \mathbf{x}_{k}\right\rangle-b_{i_{k}}\right)}{\left\|\mathbf{a}_{i_{k}}\right\|_{2}^{2}}+\frac{2\lambda\left|b_{i_{k}}^{\delta}-b_{i_{k}}\right| \cdot\left\|\mathbf{a}_{i_{k}}\right\|_{1}}{\left\|\mathbf{a}_{i_{k}}\right\|_{2}^{2}} .
\end{eqnarray*}

\noindent
Utilizing the result above,  and the proof in the case of inexact step, we have
\begin{eqnarray*}
\mathbb{E}[\|\mathbf{x}_{k+1}-\hat{\mathbf{x}}\|_2]\leq \sqrt{\prod_{i=0}^{k}q_{i} \cdot\left(2\lambda\|\hat{\mathbf{x}}\|_{1}+\|\hat{\mathbf{x}}\|_{2}^{2}\right)}+\sqrt{ \sum_{i=0}^{k}q_i\left\|\mathbf{b}^{\delta}-\mathbf{b}\right\|_{\infty}^{2} \cdot \frac{1+4\lambda\|\mathbf{ A}\|_{1,2}}{2}},
\end{eqnarray*}
which completes the proof.
\end{proof}
\textit{Remark: Comparing the error bound between SSKM method using the exact step and inexact step, we can find that inexact step can improve the performance of  SSKM method by a factor about $ \frac{1+4\lambda\|\mathbf{ A}\|_{1,2}}{2}$ compared with exact step. In the forthcoming experiments, this theoretical result will be verified.}
\section{Numerical Simulation}
\subsection{Experimental setup}
In this section, we will testify the performance of the SSKM method. The results of the tests demonstrate that SSKM method is numerically advantageous over the randomized sparse Kaczmarz method, which is denoted as SRK in this section. 

In the test, we are going to solve the linear system (\ref{problem}) by two type of different coefficient matrices $\mathbf{ A}\in\mathbb{R}^{m\times n}$. One type is the random matrix by using the MATLAB function
'randn', which produces independent standard normal entries for the matrix $\mathbf{A}$. 
Another type of matrices is  originated in different applications such as linear programming, combinatorial optimization, DNA electrophoresis model, and world city network.

In our implementations, solution $\hat{\mathbf{x}}\in\mathbb{R}^n$
is randomly generated
by using the MATLAB function "randn". The nonzero location is chosen randomly according to the sparsity. The observed data $\mathbf{b}\in\mathbb{R}^m$ is calculated by $\mathbf{A}\hat{\mathbf{x}}$.

All computations are started from the
initial vector $\mathbf{x}_0 = \bm{0}$, and terminated once the mean square error (MSE), defined by
\begin{eqnarray}
	\textrm{MSE} = \frac{\|\mathbf{x}_k-\hat{\mathbf{x}}\|_2^2}{\|\hat{\mathbf{x}}\|_2^2},
\end{eqnarray}
less than $\textrm{MSE}<10^{-6}$, or the number of iteration steps exceeds $200000$. The latter is labeled by '--' in Table \ref{Tab:1}. All of the experiments are carried out by using MATLAB(Version R2019a) on a personal computer with 2.70GHZ CPU(Intel(R) Core(TM) i7-6820HQ), 32GB memory, and Windows operating system(Windows 10).

\subsection{Parameter tuning}
In Algorithm \ref{A1}, SSKM method is affected by the parameter $\lambda$ and sampling number $\beta$.  Numerical tests will be applied to show how these two parameters influence the performance.
\begin{figure}
	\centering
	\includegraphics[width=1\textwidth]{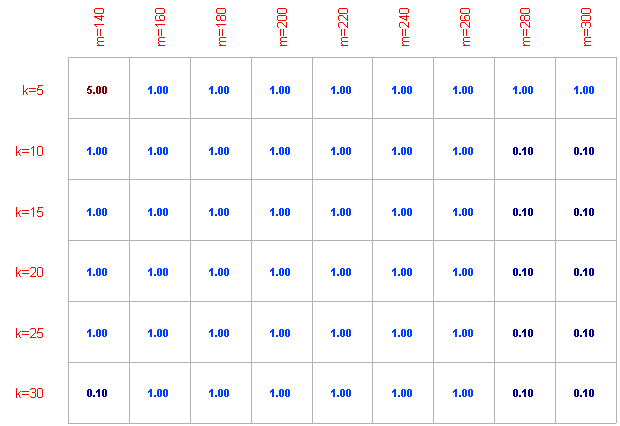}
	\caption{Given $n=200$, for each $k$ and $m$, we apply 100 tests for different $\lambda$ and show the $\lambda$ under which we can obtain the least MSE.}
   \label{lambdaimage}
\end{figure} 

\textbf{Parameter tuning for $\lambda$.}
In this test, let $n=200$, $m=140:20:300$(which means $m$ is range from 140 to 300 with interval 20), and $k=5:5:30$, the candidates of $\lambda$ are $0.01$, $0.1$, $1$, $5$, $10$. For a given $m$ and $n$, under each $\lambda$, we run experiments for $100$ times with different $\mathbf{A}$ and $\hat{\mathbf{ x}}$ by SSKM method with exact step and record the $\lambda$ having the least MSE. $\beta=m/2$ in each step. The results are shown in Fig \ref{lambdaimage}. We can find that the smaller the ratio $m/k$ is, the larger the $\lambda$ is can get a better performance. Especially, when $\lambda=1$, SSKM method can usually achieve the least MSE. As a result, we set $\lambda=1$ in our remaining tests. 

% 
%  \begin{table}
% 	\caption{The comparison between different among different $\lambda$.}
% 	\label{Tab:2}
% 	\centering
% 	\begin{tabular}{|c|c|c|c|c|c|}
% 		\hline \multicolumn{2}{|c|}{name}&$\lambda=0.01$& $\lambda=0.1$&$\lambda=1$&$\lambda=5$\\
% 		\hline
% 		\multirow{2}{*}{Inexact}&Mean MSE &0.61&0.34&$\mathbf{1.44\times10^{-5}}$&$6.35\times10^{-4}$\\
% 		\cline{2-6}        % 跨过2-4列的横线
% 		&Variation&$6.63\times10^{-12}$&$7.15\times10^{-5}$&$\mathbf{4.75\times10^{-6}}$&$1.96\times10^{-5}$\\
% 		\hline
% 		\multirow{2}{*}{Exact}&Mean MSE 
% 		&0.61&0.34&$\mathbf{2.58\times10^{-32}}$&7.06$\times10^{-32}$\\
% 		\cline{2-6}        % 跨过2-4列的横线
% 		&Variation&6.62$\times10^{-13}$&$1.40\times10^{-4}$&$\mathbf{1.15\times10^{-32}}$&$1.53\times10^{-32}$\\
% 		\hline
% 	\end{tabular}
% \end{table}
\textbf{Parameter tuning for $\beta$.}
In this test, we set $n=500$, $m=200$, $k=20$, $\lambda=1$, and the candidates of $\beta$ are chosen from $1$, $\frac{m}{4}$, $\frac{m}{2}$, $m$. $\beta_k=\beta$ for all iterations. Under each $\beta$, we also run $100$ times independent experiments. We applied SSKM method by both exact step and inexact step and record their mean MSE and standard variation. The results are shown in Table \ref{Tab:3}. We can find that the best performance of SSKM method can be achieved when $\beta=50$ for inexact step and $\beta=100$ for exact step. Note that, the results of $\beta=50$ and $\beta=100$ are the same in a level for inexact method, but for exact method, $\beta=100$ can have an obvious advantage. So, in the remaining tests, we will set $\beta=\frac{m}{2}$. 
\begin{table}
	\caption{The comparison of MSE among different $\beta$}
	\label{Tab:3}
	\centering
	\begin{tabular}{|c|c|c|c|c|c|}
		\hline \multicolumn{2}{|c|}{Name}&$\beta=1$&$\beta=50$&$\beta=100$&$\beta=200$\\
		\hline
		\multirow{2}{*}{Inexact}&Mean MSE &0.01&$\mathbf{6.28\times10^{-4}}$&$6.34\times10^{-4}$&$1.73\times10^{-3}$\\
		\cline{2-6}        % 跨过2-4列的横线
		&Variation&$1.71\times10^{-3}$&
		$\mathbf{5.40\times10^{-5}}$&$8.16\times10^{-5}$&$8.71\times10^{-15}$\\
		\hline
		\multirow{2}{*}{Exact}&Mean MSE 
		&$2.37\times10^{-5}$&$1.77\times10^{-13}$&$\mathbf{1.22\times10^{-17}}$&$6.43\times10^{-7}$\\
		\cline{2-6}        % 跨过2-4列的横线
		&Variation&$1.18\times10^{-4}$&$1.07\times10
		^{-12}$&$\mathbf{1.07\times10^{-17}}$&$8.51\times10^{-22}$\\
		\hline
	\end{tabular}
\end{table}
\subsection{Comparisons among state-of-the-arts}
In this experiment, SRK method will be compared with SSKM method. Three aspects will be considered to evaluate their performance, namely MSE, robustness, convergence rate. 

\textbf{MSE comparison.} In the test, given $n=200$, sparsity $k=5:5:30$, $m=140:20:300$. $\lambda=1$ and $\beta=\frac{m}{2}$. Under each $m$ and $k$, we record its mean MSE from 100 tests with different $\hat{\mathbf{x}}$ and $\mathbf{A}$. The results are shown in Figure \ref{MSEc}.
\begin{figure}
	\begin{subfigure}{0.5\textwidth}
		\centering
		\includegraphics[width=1\textwidth]{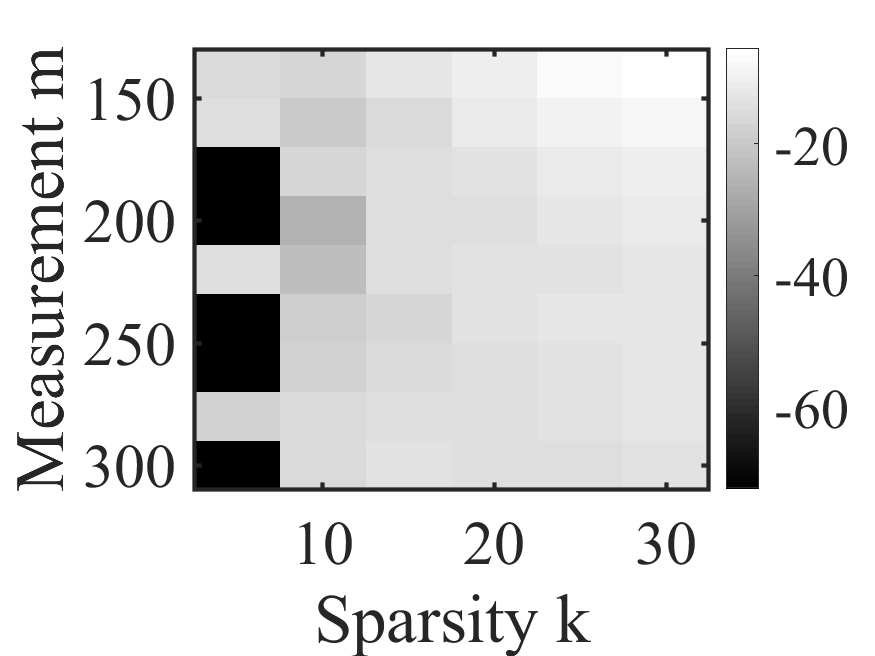}
		\caption{ Recovery results of Exact-SRK.}
	\end{subfigure}
	\begin{subfigure}{0.5\textwidth}
		\centering
		\includegraphics[width=1\textwidth]{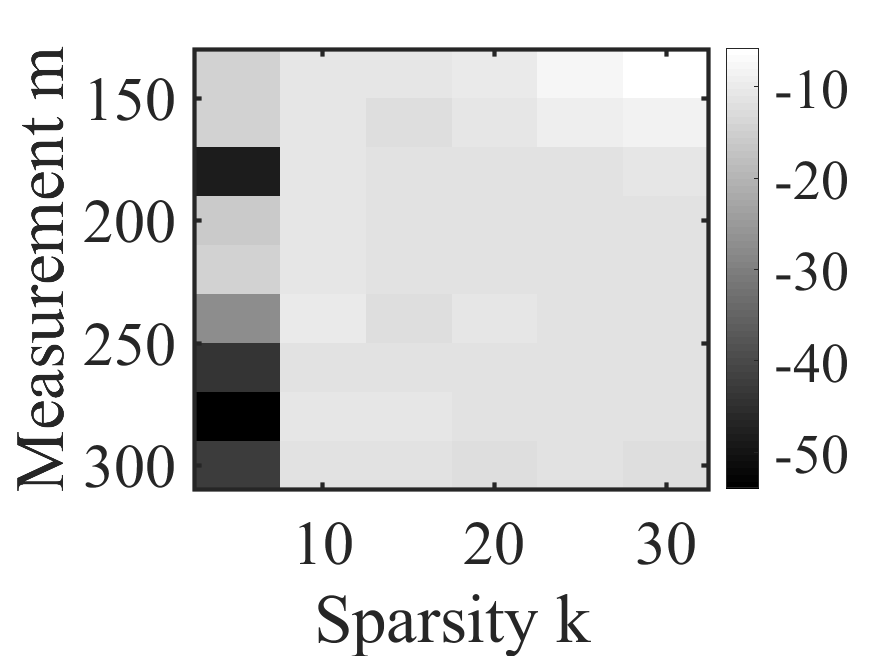}
		\caption{Recovery results of Inexact-SRK.}
	\end{subfigure}
	\begin{subfigure}{0.5\textwidth}
		\centering
		\includegraphics[width=1\textwidth]{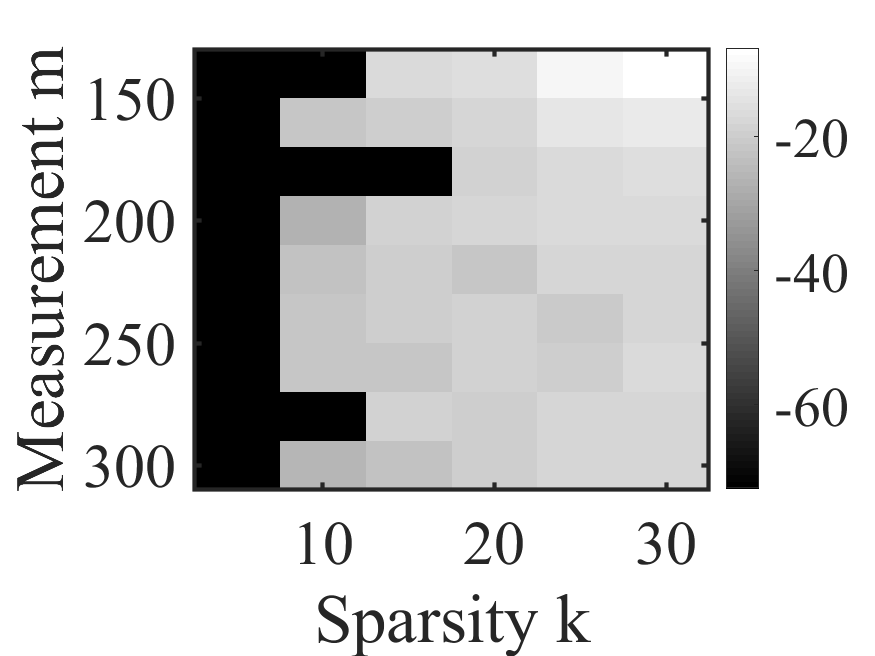}
		\caption{Recovery results of Exact-SSKM.}
	\end{subfigure}
	\begin{subfigure}{0.5\textwidth}
		\centering
		\includegraphics[width=1\textwidth]{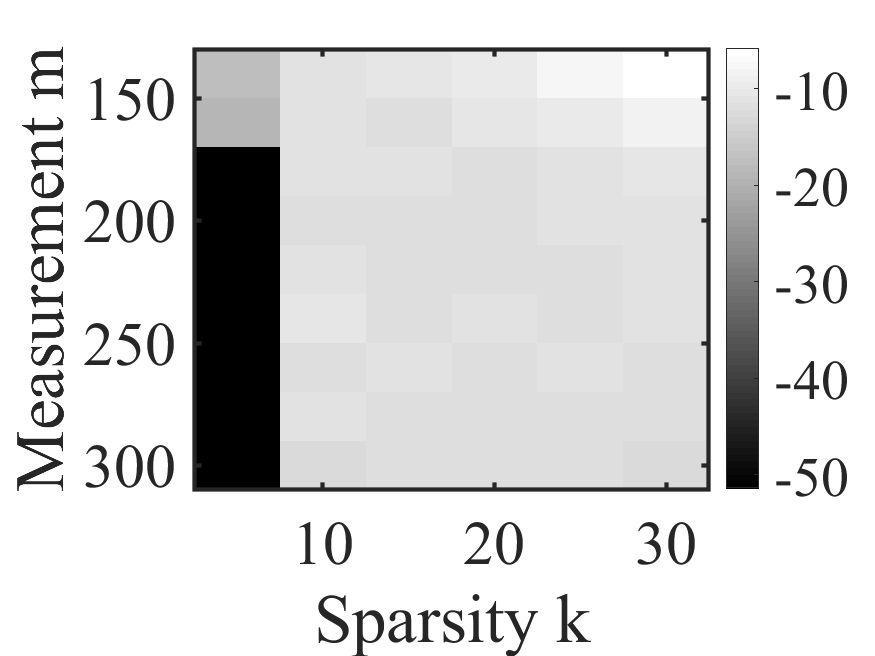}
		\caption{Recovery results of Inexact-SSKM.}
	\end{subfigure}
	\caption{Given $n=200$, testing the performance of different algorithms without noise.}\label{MSEc}
\end{figure} 

From Figure \ref{MSEc}, we can find that for both methods, MSE increases with $m/n$ when $k$ is given. When $m/n$ is fixed, the MSE decreases with sparsity $k$. This phenomena is compatible with commonsense.  Notice that, SSKM method can have a more stable performance and its corresponding MSE is lower. Moreover, we can also find that in the noiseless case, the results of SSKM method utilizing inexact step is worse than the exact step.      

\textbf{Robustness comparison.} In the test, given $n=200$, sparsity $k=5:5:30$, $m=140:20:300$, $\lambda=1$ and $\beta=\frac{m}{2}$. The level of the measurement noise is 10\%. For each $m$ and $k$, we record its mean MSE by 100 tests with different $\hat{\mathbf{x}}$ and $\mathbf{A}$. The results are shown in Figure \ref{Noisec}.
\begin{figure}
	\begin{subfigure}{0.5\textwidth}
		\centering
		\includegraphics[width=1\textwidth]{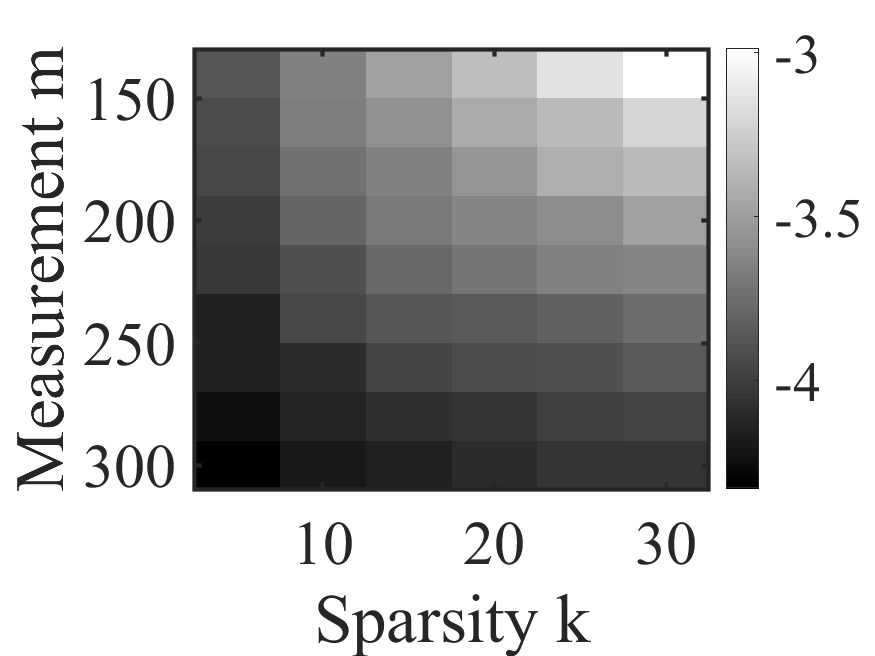}
		\caption{Recovery results of Exact-SRK.}
	\end{subfigure}
	\begin{subfigure}{0.5\textwidth}
		\centering
		\includegraphics[width=1\textwidth]{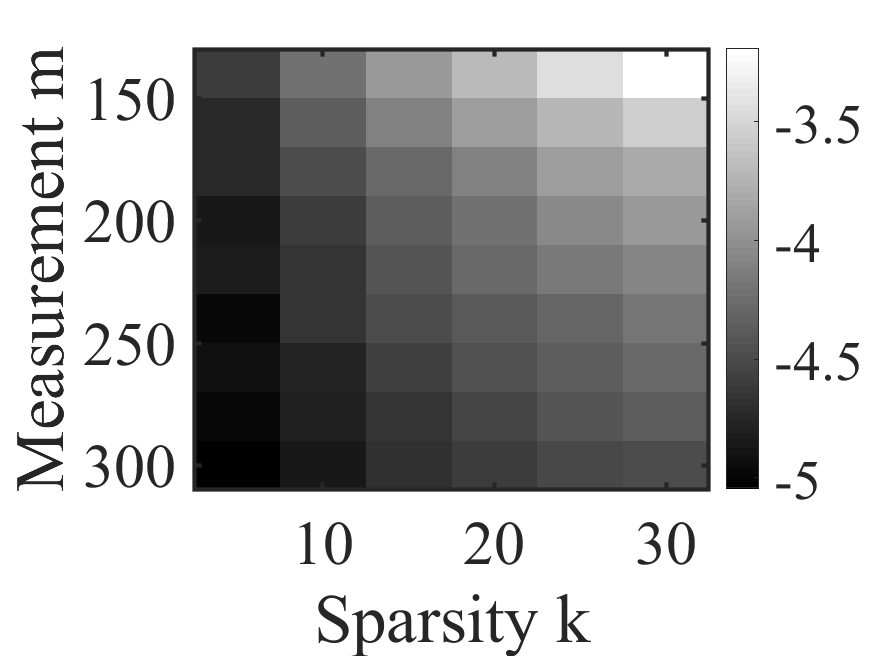}
		\caption{Recovery results of Inexact-SRK.}
	\end{subfigure}
	\begin{subfigure}{0.5\textwidth}
		\centering
		\includegraphics[width=1\textwidth]{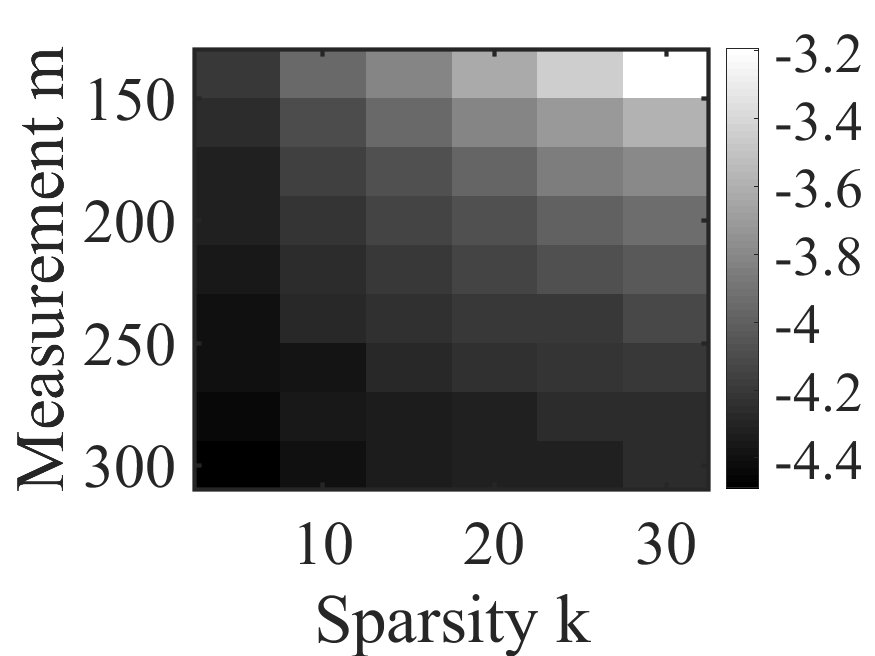}
		\caption{Recovery results of Exact-SSKM.}
	\end{subfigure}
	\begin{subfigure}{0.5\textwidth}
		\centering
		\includegraphics[width=1\textwidth]{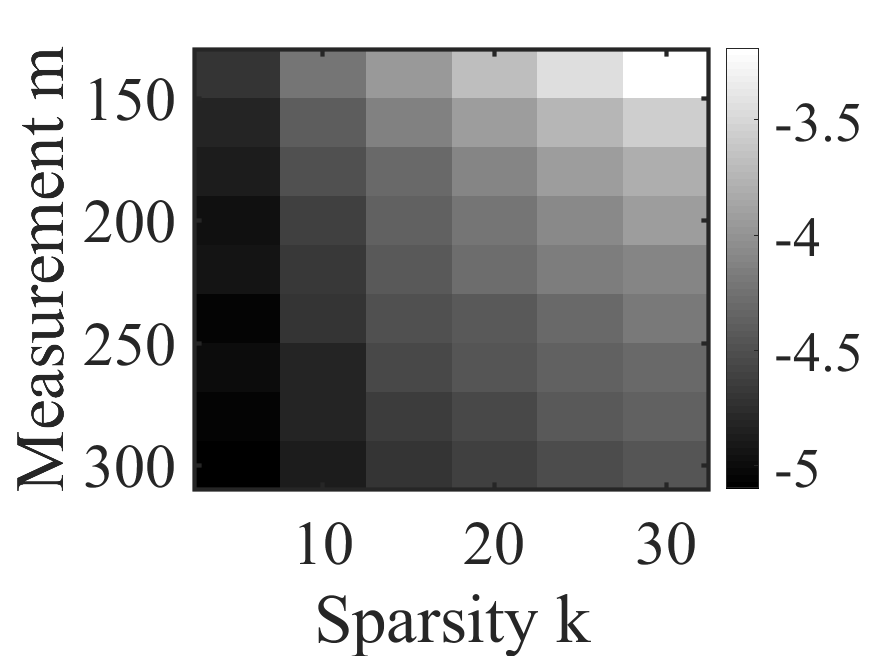}
		\caption{Recovery results of Inexact-SSKM.}
	\end{subfigure}
	\caption{Given $n=200$, testing the performance of different algorithms with the noise}
	\label{Noisec}
\end{figure} 

From Figure \ref{Noisec}, we can find that both methods can resist the noise in some degree. When $m/k$ is larger, we can get a better performance. On the contrary, SSKM method utilizing inexact step performs better than exact step. In \cite{2016Linear}, similar phenomena can also be noticed. It is caused to the corruption by the noise, which makes the estimation calculated by exact step have larger bias. Thus, the efficiency of the exact step becomes lower. 

\textbf{Convergence speed comparison.} In this test, we only consider the noiseless condition. Given $n=200$, sparsity $k=30$, $m$ is chosen from $150, 200, 300$, $\lambda=1$ and $\beta=\frac{m}{2}$. For each $m$, we record its MSE from 100 tests with different $\hat{\mathbf{x}}$ and $\mathbf{A}$. The results are shown in Figure \ref{Convergencerate}.

From Figure \ref{Convergencerate}, we can find that SSKM method can have a faster convergence rate and lower MSE than SRK method.  This phenomena is due the greedy sampling rule, which picks out the most violated row. At the same time, we can also find that SSKM method utilizing exact step can have a faster convergence speed than utilizing inexact step.
\begin{figure}
	\begin{subfigure}{0.5\textwidth}
		\centering
		\includegraphics[width=1\textwidth]{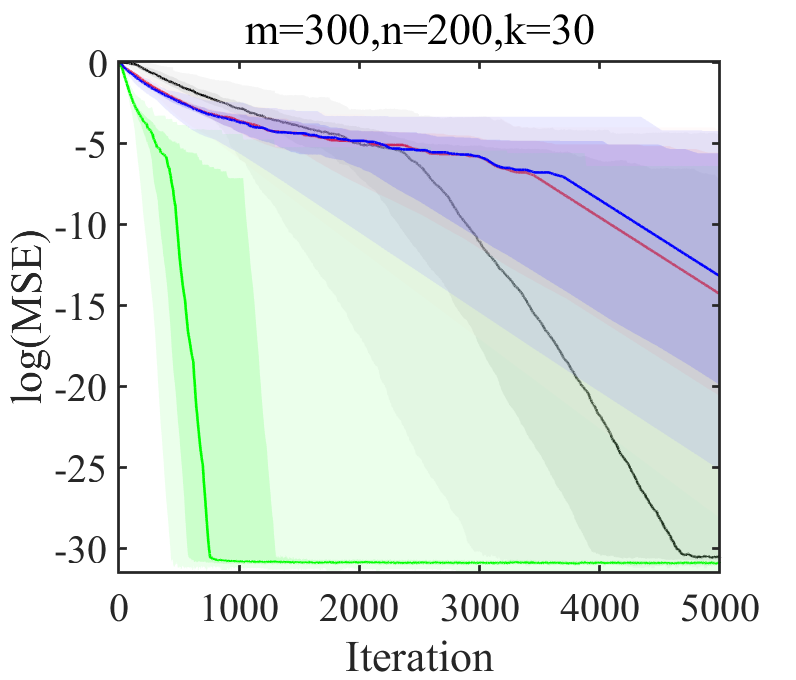}
		\caption{}
	\end{subfigure}
	\begin{subfigure}{0.5\textwidth}
		\centering
		\includegraphics[width=1\textwidth]{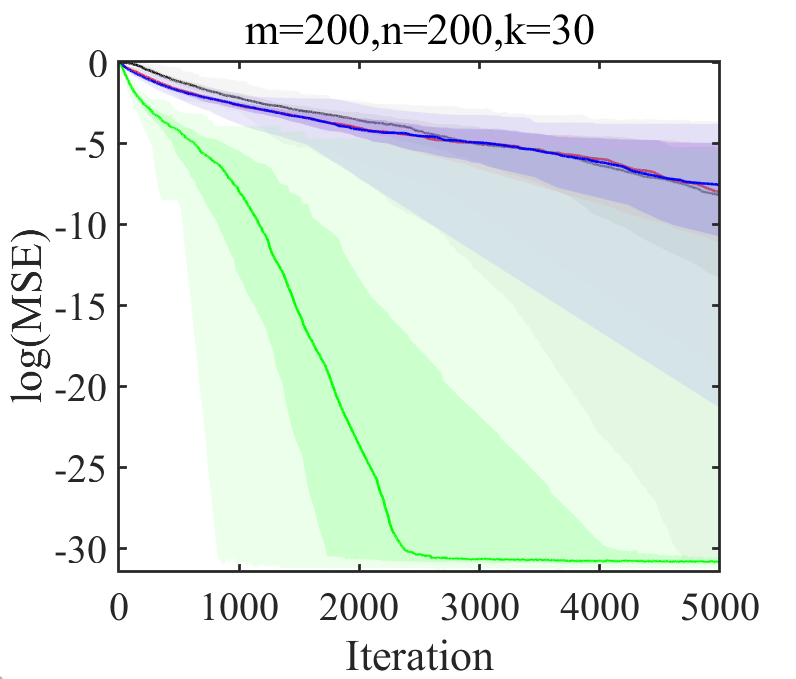}
		\caption{}\label{setup}
	\end{subfigure}
	\begin{subfigure}{1\textwidth}
		\centering
		\includegraphics[width=0.66\textwidth]{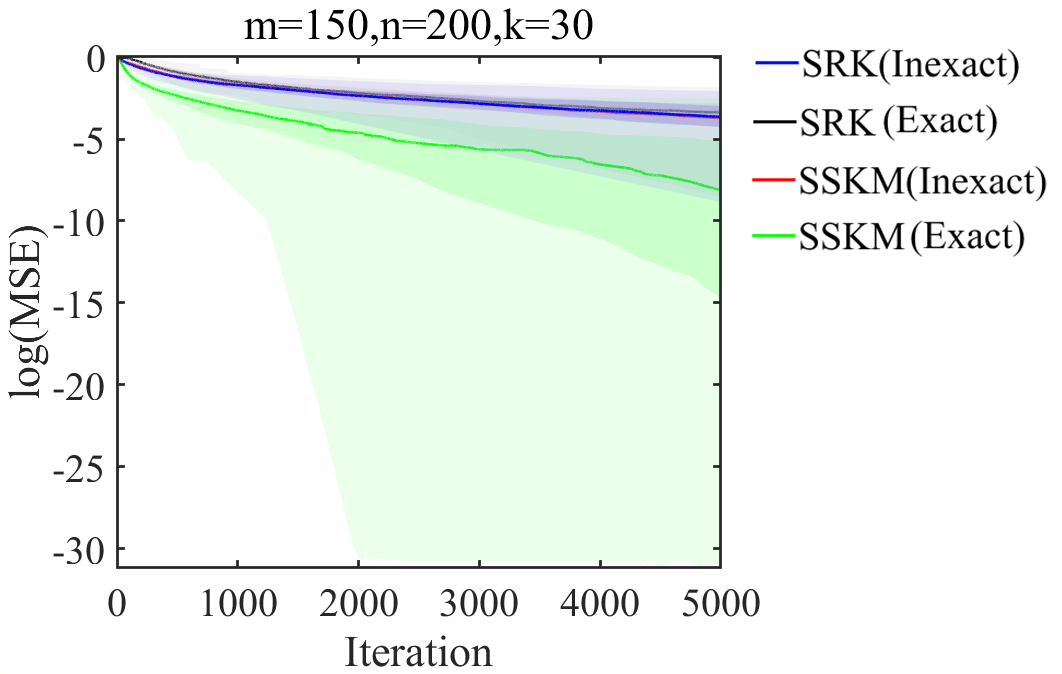}
		\caption{}
	\end{subfigure}
	\caption{Comparison of different methods for convergence speed. Thick line shows median over 100 trials, light area is between minimum and maximum, and darker area indicates 25th and 75th quantile(color figure online).}\label{Convergencerate}
\end{figure} 

\subsection{Real data}
~

\textbf{Real matrix.}
In this simulation, we will verify the performance of SSKM method under some real matrices\cite{realmatrix}. The matrices are collected in \textit{The university of florida sparse matrix collection} which are originated in different kinds of applications. In this test, under each matrix $\mathbf{A}$, we generate 100 different ground truth $\hat{\mathbf{x}}$ and calculate $\mathbf{A}\hat{\mathbf{x}}$ as observed data $\mathbf{b}$. 

CPU and IT mean the arithmetical average of the elapsed CPU
times and the required iteration steps once the MSE is below $10^{-6}$ with respect to 100 times repeated runs of the
corresponding method.
$\text{Cond}(A)$ refers to the condition number of $\mathbf{A}$, and the density of $\mathbf{ A}$ is also defined as
\begin{eqnarray*}
	\text{Density}:=\frac{\text{number of nonzeros of $\mathbf{ A}$}}{m\times n}.
\end{eqnarray*}.  

\noindent We only make comparison in the noiseless case. The SSKM and SRK method utilize exact step. The results are shown in Table \ref{Tab:1}.

From Table \ref{Tab:1}, we can find out that the SSKM method can achieve the best performance. It requires least iterations and CPU time to terminate. Although the sampling rule of the SSKM method demands a larger burden of time at each iteration, the overall iteration times of it can remedy this drawback and achieve a fewer CPU time.  
\begin{table}
	\caption{The results of different methods dealing with the real data.}
	\label{Tab:1}
	\centering
	\begin{tabular}{|c|c|c|c|c|c|c|}
		\hline \multicolumn{2}{|c|}{Name}&model1& Trefethen\_300&WorldCities&Trefethen\_20&flower\_5\_1\\
		\hline  \multicolumn{2}{|c|}{$m\times n$}&$362\times798$&$300\times300$&$315\times 100$&$20\times20$&$211\times201$\\
		\hline
		\multicolumn{2}{|c|}{Density} &0.34\%&5.20\%&23.87\%&39.50\%&1.42\%\\
		\hline  \multicolumn{2}{|c|}{Cond($\mathbf{ A}$)}&17.57&1772.69&66.00&63.09&Inf\\
		\hline
		\multicolumn{2}{|c|}{Sparsity}&20&20&20&20&20\\
		\hline
		\multirow{2}{*}{RK}&IT &--&--&2977.9&11886&--\\
		\cline{2-7}        % 跨过2-4列的横线
		&CPU&--&--&0.41&16.78&--\\
		\hline
		\multirow{2}{*}{SRK}&IT &6023.3&11213&9870.4&27783&4519.8\\
		\cline{2-7}        % 跨过2-4列的横线
		&CPU&0.58&0.56&0.29&0.62&0.17\\
		\hline
		\multirow{2}{*}{SSKM}&IT 
		&884.41&2560.2&2575.5&9395.6&864.72\\
		\cline{2-7}        % 跨过2-4列的横线
		&CPU&0.47&0.49&0.69&0.28&0.15\\
		\hline
	\end{tabular}
\end{table}

\textbf{Phantom picture.}
In this test, we will study an academic tomography problem. The underlying model in the test problem consists of straight X-rays which penetrate the object, afterwards the damping is recorded. According to Lambert-Beer’s law, and after taking the logarithm of the recorded data, the damping is given as a line integral along the X-ray of the object’s attenuation coefficient, which is formulated as a linear equations model $\mathbf{A}\mathbf{x}=\mathbf{b}$.

%The object domain is the square $[-N/2, N/2]\times[-N/2, N/ 2]$.The detector has $p$ pixels, and the assumption is that each detector pixel is hit by a single X-ray. Moreover, the source-detector pair is rotated around the object, and measurements are recorded for $N_\theta$ angles $\theta_1,\theta_2,\cdots,\theta_N$. Hence, the number of data is $m = pN_\theta$. The Paralleltomo will be considered in this test where the source is placed infinitely far from a flat detector with equal spacing between the pixels (which is a very good approximation to the situation at synchrotron facilities) and the p rays are therefore parallel. In the test, the number of rays is 75, the projections angles are $1^{\circ},6^{\circ},\cdots,178^{\circ}$

%\begin{enumerate}
%	\item \textbf{Paralleltomo}. The source is placed infinitely far from a flat detector with equal spacing between the pixels (which is a very good approximation to the situation at synchrotron facilities) and the p rays are therefore parallel. In the test, the number of rays is 75, the projections angles are $1^{\circ},6^{\circ},\cdots,178^{\circ}$
%%	\item \textbf{fancurvedtomo}. The source is located at distance $RN$ from the object’s center. The detector is curved such that there is an equal angular span between each ray.
%%	\item \textbf{fanlineartomo}. The source is also located at distance RN from the object’s center. The detector is linear so that there is an uneven angular span between each ray.
%\end{enumerate}

We used the AIRtools toolbox\cite{2017AIR} to generate the matrix $\mathbf{A}$. In this test, $n=2500$, $m=2049$. The image of interest is shepplogan shown in Figure \ref{imagedata}(a), which is sparse; 
thus we can apply sparse kaczmarz method to recover it from dump. In this test, we compare SSKM with SRK method. The results are shown in Figure \ref{imagedata}. Here, both SSKM and SRK method use the exact step in the test.
\begin{figure}
	\begin{subfigure}{0.32\textwidth}
		\centering
		\includegraphics[width=1\textwidth]{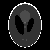}
		\caption{Original.}
	\end{subfigure}
	\begin{subfigure}{0.32\textwidth}
		\centering
		\includegraphics[width=1\textwidth]{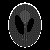}
		\caption{Recovered by SSKM.}
	\end{subfigure}
	\begin{subfigure}{0.32\textwidth}
		\centering
		\includegraphics[width=1\textwidth]{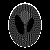}
		\caption{Recovered by SRK.}
	\end{subfigure}
	\begin{subfigure}{1\textwidth}
		\centering
		\includegraphics[width=1\textwidth]{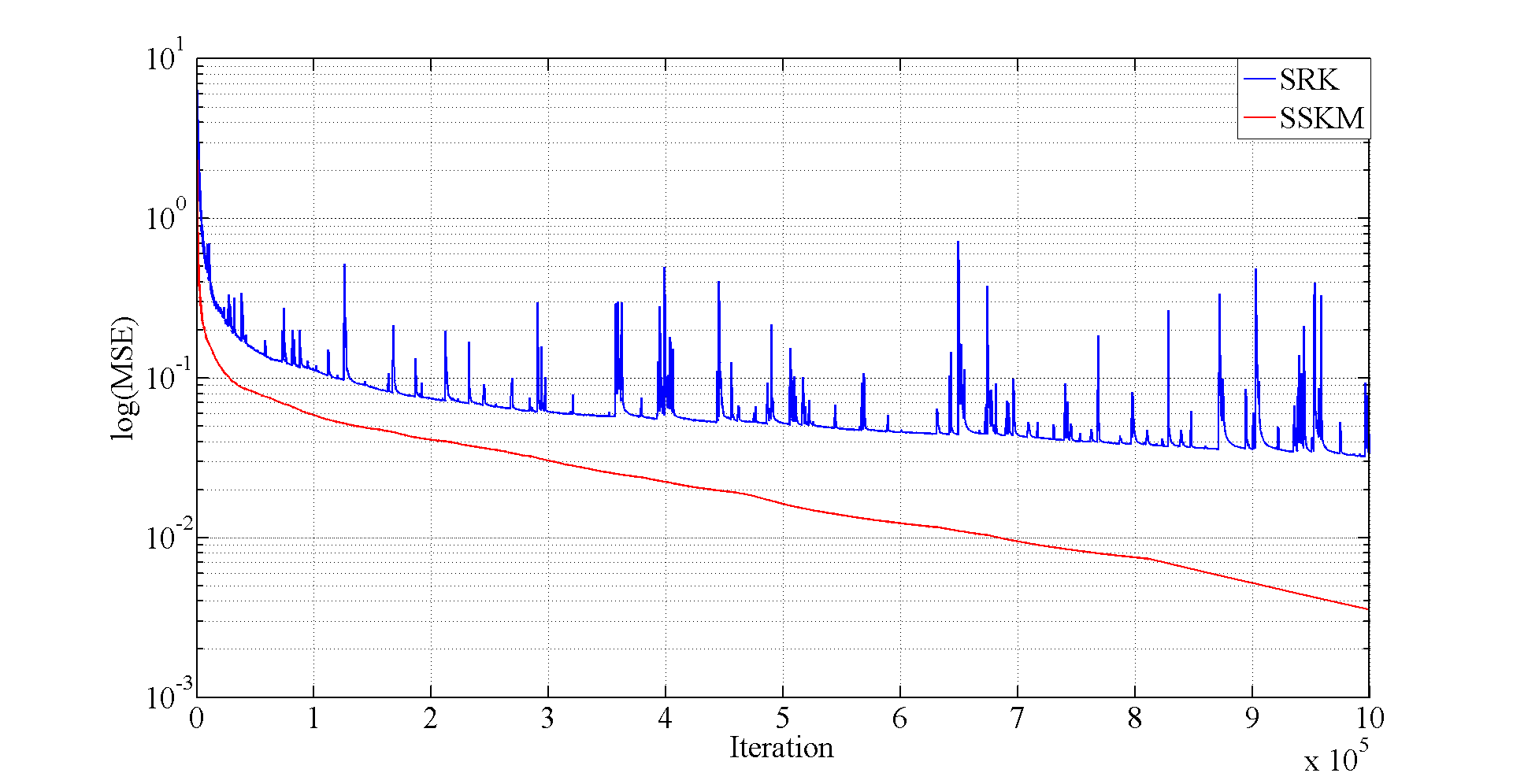}
		\caption{The curves of MSE for SSKM and SRK method.}
	\end{subfigure}
	\caption{Experiment results. (a) is the original picture. (b) is the result recovered by SSKM method.(c) is the result recovered by the SRK method. (d) is the error curve.}\label{imagedata}
\end{figure} 

From Figure \ref{imagedata}, we can find that the SSKM method has advantages over the SRK by the quality of the recovered image. Although the matrix is of rank deficiency, by utilizing the sparsity structure of image, we can still find the ground truth. It will shed on light on more applications to utilize prior information to recover signal of interest with fewer measurements. 

\section{Conclusion}
In this paper, we introduce the SSKM method to find the sparse solutions of linear systems. It combines the Bregman projection and the Sampling Kaczmarz-Motzkin method. The former helps us to find the sparse solution implicitly; The latter is employed to accelerate the convergence rate of the method. Theoretically, we prove linear convergence of the SSKM method for the noiseless and noisy cases respectively. Numerical tests from both simulations and applications demonstrate the effectiveness of the proposed method.

\section{Acknowledgment}
The paper is granted by the National Natural Science Foundation:11971480, 61977065, and the Hunan Province excellent youngster Foundation:2020JJ3038.  
\bibliographystyle{unsrt}
\bibliography{references}

\end{document}